\documentclass[a4paper,oneside,12pt]{article}

\usepackage{amsmath,amsthm,amsfonts,amscd,amssymb}
\usepackage{longtable,geometry}
\usepackage[english]{babel}
\usepackage[utf8]{inputenc}
\usepackage[active]{srcltx}
\usepackage[T1]{fontenc}
\usepackage{graphicx}
\usepackage{pstricks}
\usepackage{bbm}
\usepackage{mathtools}

\usepackage{MnSymbol}
\usepackage{stmaryrd}
\usepackage{nicefrac}
\usepackage{calrsfs}

\usepackage{rotating}
\usepackage{xcolor}
\usepackage{framed}
\usepackage{hyperref}
\definecolor{colorlinks}{RGB}{0, 24, 168}
\definecolor{colorcites}{RGB}{124, 10, 2}
\hypersetup{
	colorlinks=true,
	linkcolor=colorlinks,
	citecolor=colorcites,
	urlcolor=colorlinks,
	pdfborder={0 0 0}
}

\usepackage{stackrel}

\usepackage{enumitem}
\usepackage{cleveref}

\usepackage{subfigure}

\colorlet{shadecolor}{blue!15}

\newtheorem{thm}{Theorem}

\newtheorem{prop}{Proposition}[section]
\newtheorem{cor}[prop]{Corollary}
\newtheorem{lem}[prop]{Lemma}

\newtheorem*{rem}{Remark}

\newcommand{\be}[1]{\begin{equation}\label{#1}}
\newcommand{\ee}{\end{equation}}
\numberwithin{equation}{section}

\newcommand{\ba}[1]{\begin{align}\label{#1}}
\newcommand{\ea}{\end{align}}
\numberwithin{equation}{section}

\newcommand{\ben}{\begin{equation*}}
\newcommand{\een}{\end{equation*}}
\numberwithin{equation}{section}

\newcommand{\calC}{\mathcal{C}}
\newcommand{\calD}{\mathcal{D}}
\newcommand{\calE}{\mathcal{E}}

\newcommand{\calL}{\mathcal{L}}

\newcommand{\calS}{\mathcal{S}}
\newcommand{\calT}{\mathcal{T}}

\newcommand{\calV}{\mathcal{V}}

\newcommand{\bbE}{\mathbb{E}}

\newcommand{\bbH}{\mathbb{H}}

\newcommand{\bbN}{\mathbb{N}}

\newcommand{\bbP}{\mathbb{P}}

\newcommand{\bbT}{\mathbb{T}}
\newcommand{\bbU}{\mathbb{U}}

\newcommand{\bbZ}{\mathbb{Z}}
\newcommand{\eps}{\epsilon}
\newcommand{\dist}{\mathrm{dist}}

\newcommand{\rk}[1]{\bgroup\color{red}%
  \par\medskip\hrule\smallskip%
  \noindent\textbf{#1}%
  \par\smallskip\hrule\medskip\egroup}

\renewcommand{\int}{\mathrm{in}}

\newcommand{\zero}{\mathsf{0}}
\newcommand{\1}{1}

\newcommand{\XOR}{\,\oplus \,}

\newcommand{\tor}{\mathsf{Tor}}
\renewcommand{\colon}{\,:\,}

\usepackage{enumitem}
\setlist[itemize]{itemsep=1pt, topsep=4pt}
\setlist[enumerate]{itemsep=1pt, topsep=4pt}

\title{Macroscopic loops in the loop~$O(n)$ model via the XOR trick}
\date{\today}

\author{Nicholas Crawford\thanks{The Technion, Israel. \url{nickc@tx.technion.ac.il}} \and Alexander Glazman\thanks{University of Innsbruck, Austria. \url{alexander.glazman@uibk.ac.at}}  \and Matan Harel\thanks{Northeastern University, USA. \url{m.harel@northeastern.edu}}  \and Ron Peled\thanks{Tel Aviv University, Israel. \url{peledron@tauex.tau.ac.il}}}
\date{\today}

\begin{document}

\maketitle

\begin{abstract}
The loop $O(n)$ model is a family of probability measures on collections of non-intersecting loops on the hexagonal lattice, parameterized by a loop-weight~$n$ and an edge-weight~$x$. 
Nienhuis predicts that, for $0 \leq n \leq 2$, the model exhibits two regimes separated by~$x_c(n) = 1/\sqrt{2 + \sqrt{2-n}}$: when $x < x_c(n)$, the loop lengths have exponential tails, while, when $x \geq x_c(n)$, the loops are macroscopic.

In this paper, we prove three results regarding the existence of long loops in the loop $O(n)$ model:\\
-- In the regime $(n,x) \in [1,1+\delta) \times (1- \delta, 1]$ with $\delta >0$ small, a configuration sampled from a translation-invariant Gibbs measure will either contain an infinite path or have infinitely many loops surrounding every face. In the subregime $n \in [1,1+\delta)$ and $x \in (1-\delta,1/\sqrt{n}]$ our results further imply Russo--Seymour--Welsh theory. This is the first proof of the existence of macroscopic loops in a positive area subset of the phase diagram.\\
-- Existence of loops whose diameter is comparable to that of a finite domain whenever $n=1, x \in (1,\sqrt{3}]$; this regime is equivalent to part of the antiferromagnetic regime of the Ising model on the triangular lattice.\\
-- Existence of non-contractible loops on a torus when $n \in [1,2], x=1$.

The main ingredients of the proof are: (i) the `XOR trick': if $\omega$ is a collection of short loops and $\Gamma$ is a long loop, then the symmetric difference of $\omega$ and $\Gamma$ necessarily includes a long loop as well; (ii) a reduction of the problem of finding long loops to proving that a percolation process on an auxiliary planar graph, built using the Chayes--Machta and Edwards--Sokal geometric expansions, has no infinite connected components; and (iii) a recent result on the percolation threshold of Benjamini--Schramm limits of planar graphs.
\end{abstract}

\section{Introduction}\label{sec:intro}

The loop~$O(n)$ model is a model for non-intersecting simple cycles (which we term {\em loops}) on the hexagonal lattice~$\bbH$, parameterized by a loop weight $n > 0$ and an edge weight $x>0$, and defined as follows:
A \emph{loop configuration} is a spanning subgraph of $\bbH$ in which
every vertex has even degree (see Figure~\ref{fig:Loop O n model}). Note that a loop configuration can a priori consist of loops (i.e., subgraphs which are isomorphic to a cycle) together with isolated vertices and bi-infinite paths.
For a subgraph $\calD$ of the hexagonal lattice $\bbH$ and a loop configuration $\xi$, let $\calE(\calD,\xi)$ be the set of loop configurations coinciding with $\xi$ outside $\calD$. The loop $O(n)$ measure on $\calD$ with edge-weight $x$ and boundary conditions $\xi$ is the probability measure
$\bbP_{\calD,n,x}^\xi$ on $\calE(\calD,\xi)$ defined by the formula
  \[
  \bbP_{\calD,n,x}^\xi(\omega) := \frac{x^{|\omega|} n^{\ell(\omega)}}{Z_{\calD,n,x}^\xi},  \]
for every $\omega\in \calE(\calD,\xi)$,  where $|\omega|$ is the number of edges of $\omega\cap\calD$, $\ell(\omega)$ is the number of loops or bi-infinite paths of $\omega$ intersecting $\calD$, and $Z_{\calD,n,x}^\xi$ is the unique constant making $\bbP_{\calD,n,x}^\xi$ a probability measure.
We may also extend the measure to~$n=0$ or~$x=\infty$ by taking the appropriate limits.

\begin{figure}[t]
\centering 
	\subfigure[n=1.4, x=0.6]{
	\includegraphics[scale=0.13]{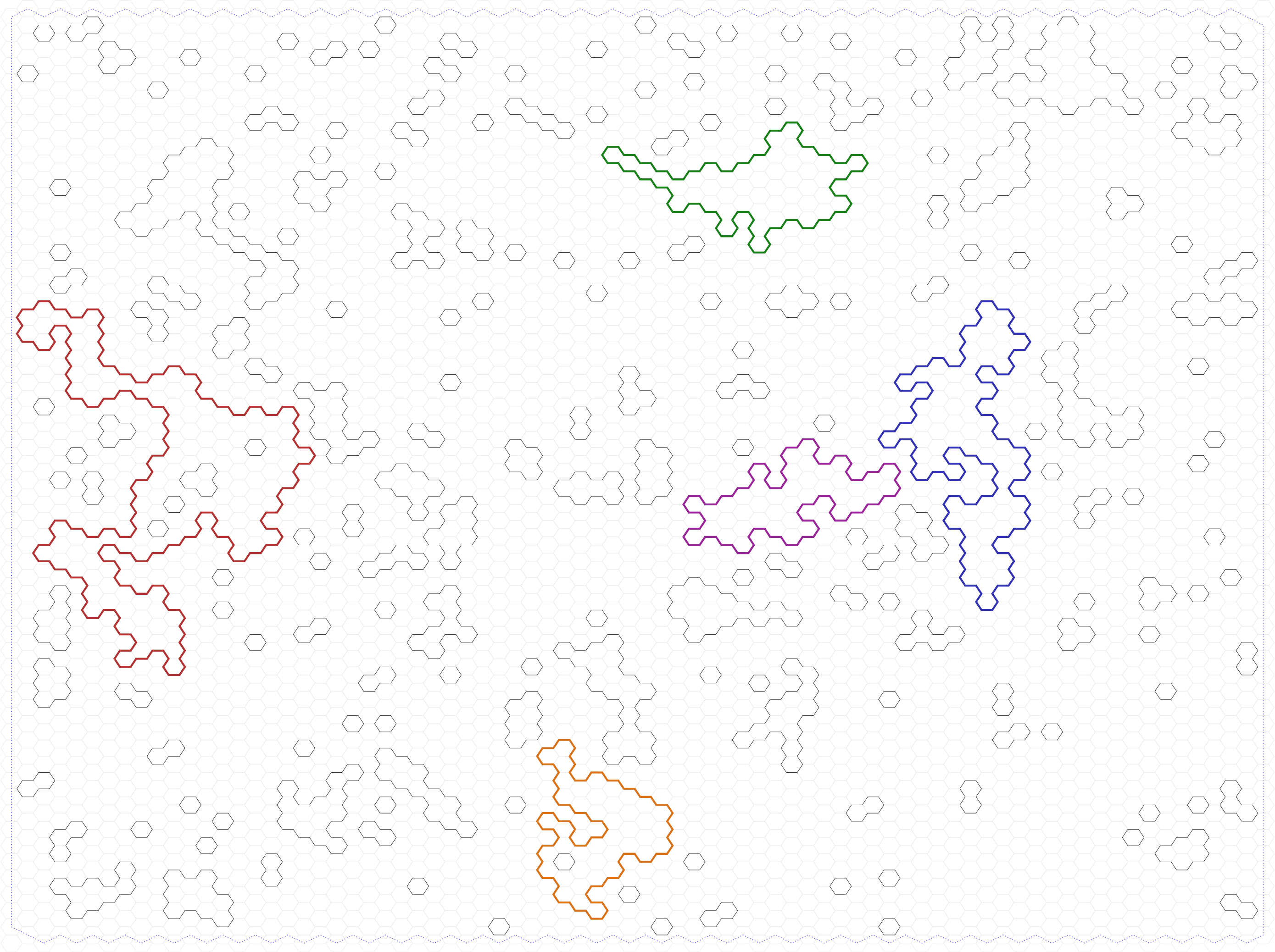}
	\label{fig:subfig2} } \subfigure[n=1.4, x=0.63]{
	\includegraphics[scale=0.13]{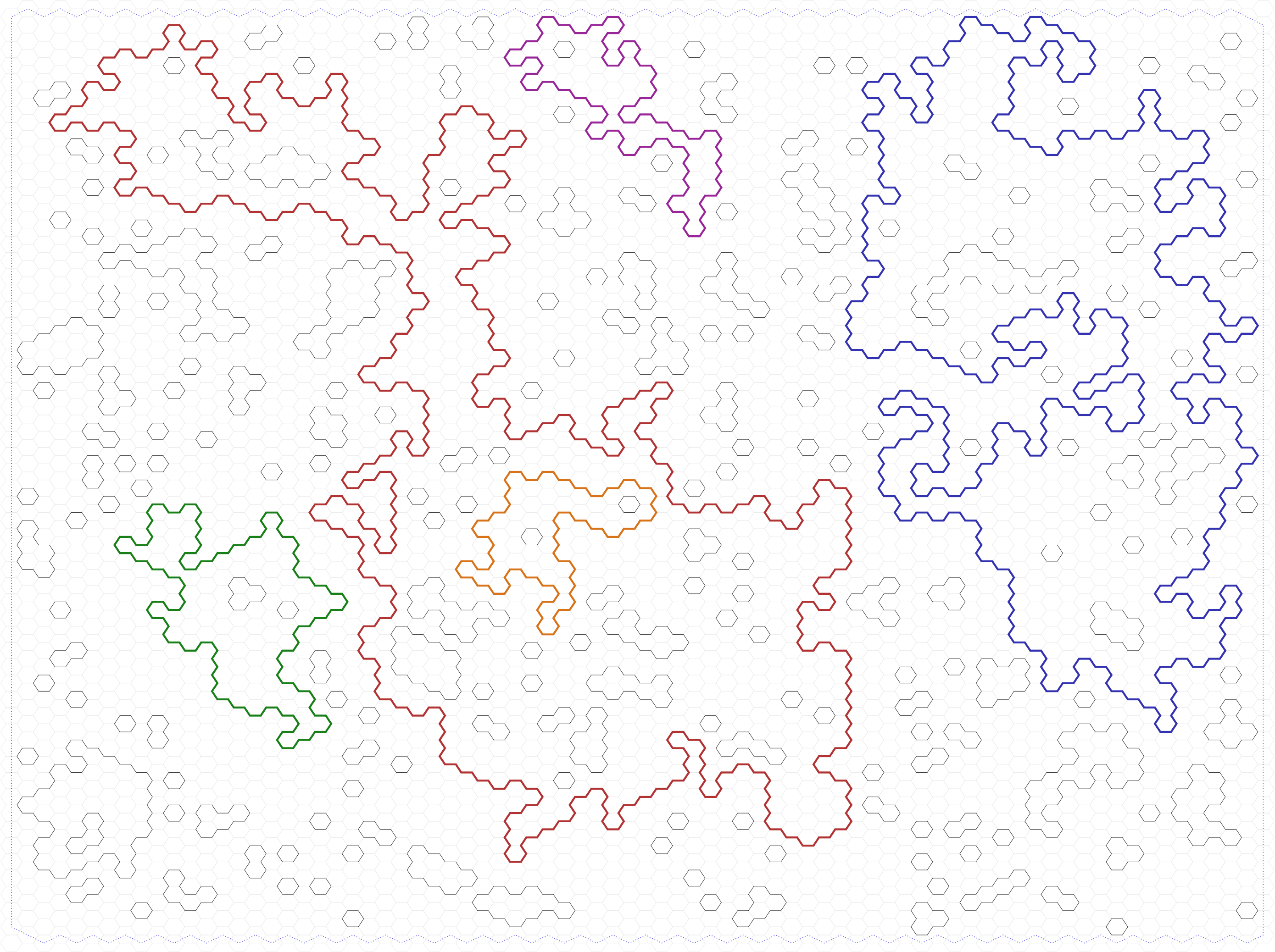}
	\label{fig:subfig3} } \caption{Samples of the loop $O(n)$ model for two nearby parameter values: the first in the predicted exponential decay regime and the second in the predicted macroscopic loops regime. The longest loops are colored for visibility. \label{fig:Loop O n model}}
\end{figure}

\smallskip
{\bf Background.} The loop~$O(n)$ model contains other notable models of statistical mechanics as special cases: the Ising model ($n=1$), critical percolation ($n=x=1$), the dimer model ($n=1, x=\infty$), self-avoiding walk ($n=0$), Lipschitz functions ($n=2$), proper 4-colorings ($n=2,x=\infty$), dilute Potts $(n =\sqrt{q}, q $ integer), and the hard-hexagon model ($n \to \infty, nx^6 \to \lambda$). Furthermore, the model serves as an approximate graphical representation of the spin $O(n)$ model, conjectured to be in the same universality class, which was the original motivation for its introduction~\cite{DomMukNie81}. See~\cite{PelSpi17} for a recent survey on both~$O(n)$ models.

 A tantalizing 1982 prediction of Nienhuis~\cite{Nie82}, with later refinements~\cite{Nie84, BloNie89},~\cite[Section 5.6]{KagNie04},~\cite[Section 2.2]{Smi06}, states that whenever
\begin{equation}\label{eq:Nienhuis line}
  0< n \le 2\quad\text{and}\quad x\ge x_c(n)=\frac{1}{\sqrt{2 + \sqrt{2-n}}},
\end{equation}
the loop~$O(n)$ model has a conformal-invariant scaling limit~$\mathrm{CLE}_{\kappa}$, 
where~$\kappa = \kappa(n,x)$ may take all values in~$(8/3,8)$.
In particular, in this regime, the loops are {\em macroscopic}: for any domain and any boundary conditions, there is a loop surrounding a ball of radius comparable to that of the domain with a uniformly positive probability.
For all other parameter values, the length of the loop passing through a given vertex is expected to have exponential tails, uniformly in the domain and vertex (for empty boundary conditions).

Mathematical progress on the predictions remains limited. Conformal invariance has only been obtained for the critical Ising model ($n=1, x=\frac{1}{\sqrt{3}}$)~\cite{Smi10,CheSmi11,CheDumHonKemSmi14} and for critical site percolation on the triangular lattice ($n=x=1$)~\cite{Smi01, CamNew06}. Other progress concerns coarser properties of the loop structure, as summarized in Figure~\ref{fig:phase diagram} and detailed below.

\begin{figure}[t]
	\begin{center}
		\includegraphics[scale=1.2]{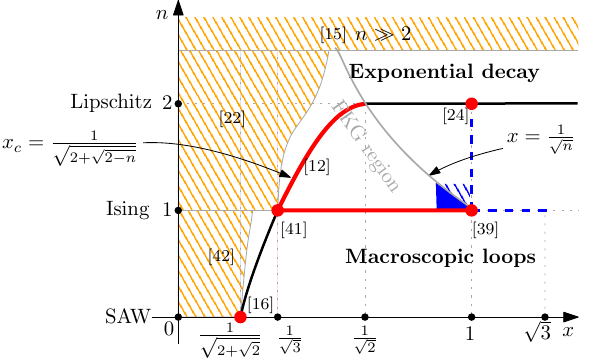}
	\end{center}
	\caption{The predicted phase diagram for the loop $O(n)$ model. The critical line $x_c$ separating the regime of exponential decay from the regime of macroscopic loops is plotted in bold. The region~$n\geq 1, x\leq \tfrac{1}{\sqrt{n}}$ where a dichotomy between the two behaviors is proved is denoted FKG region. Orange lines show regions where exponential decay is proved. Red dots or lines mark regions where macroscopic loops are proven to occur. The results of this paper are indicated in blue: macroscopic loops are established in the filled area, while exponential decay is ruled out in the dashed area and segments. Picture adapted from~\cite{GlaMan21b}.}
	\label{fig:phase diagram}
\end{figure}

As mentioned above, loop lengths are predicted to follow one of two types of behaviors, according to the value of $n$ and $x$: either macroscopic loops appear, or the length of loops has exponential tail decay. However, other behaviors have not been ruled out in general. Recently, such a dichotomy has been established in the parameter range $n \geq 1, x \leq \tfrac{1}{\sqrt{n}}$~\cite{DumGlaPel21}, through the use of a positively associated representation of the model.
%where an associated spin representation is positively associated (FKG).

The only two cases where the loop $O(n)$ model was shown to exhibit a phase transition at $x=x_c(n)$ are $n=1$ (Ising model) and $n=0$ (self-avoiding walk). Exponential tails for loop lengths in the low-temperature Ising model ($x<1/\sqrt{3}$)~\cite{Ons44,AizBarFer87} and existence of infinitely many loops around each vertex (in the unique Gibbs measure) for the critical and high-temperature ferromagnetic Ising model ($1/\sqrt{3}\leq x\leq 1$) are classical. In the latter case, emergence of macroscopic loops follows from a general Russo--Seymour--Welsh (RSW) theory developed in~\cite{Tas16} (or, alternatively, from the dichotomy result~\cite{DumGlaPel21}).\footnote{An extension of the RSW theory to~$1<x<1+\varepsilon$ is discussed in a preprint ``Percolation without FKG'' by Beffara and Gayet from~2017. However, they have communicated to us that their proof contains a crucial mistake.}
The critical point of the self-avoiding walk on the hexagonal lattice was proven to equal~$\tfrac{1}{\sqrt{2+\sqrt{2}}}$ in the celebrated work~\cite{DumSmi12}. The walk was shown to scale to a straight line segment for smaller $x$~\cite{Iof98}, and to be space-filling for larger $x$~\cite{DumKozYad11}.

The results known beyond the cases $n=0,1$ are as follows. Exponential decay was established when either $n>0, x < \tfrac{1}{\sqrt{2+\sqrt{2}}} + \varepsilon(n)$~\cite{Tag18} or $n >1,  x < \tfrac{1}{\sqrt{3}}+ \varepsilon(n)$~\cite{GlaMan21b}, by comparing to the behavior of the self-avoiding walk and Ising model, respectively. In addition, for sufficiently large $n$, exponential decay was proven for all~$x>0$ (and an ordering transition was further established)~\cite{DumPelSamSpi17}. Lastly, existence of macroscopic loops (as well as Russo--Seymour--Welsh type estimates) was recently shown to occur on the line $x=x_c(n)$ for $1\le n\le 2$~\cite{DumGlaPel21} and also at $n=2, x=1$~\cite{GlaMan21} (uniform Lipschitz functions).

\smallskip
{\bf Results.} Existence of macroscopic loops has been established only in the rather sparse set of parameters discussed above. In fact, away from this set, even the more modest goal of excluding exponential decay of loop lengths has not been achieved. The goal of the present work is to introduce a new technique for showing the existence of long loops in the loop $O(n)$ model. The technique applies in the vicinity of the critical percolation point $n=x=1$ and allows us to derive the following three results:
\begin{itemize}
	\item[\emph{(i)}] For some $\delta>0$, the model with $1\le n \leq 1+\delta$ and $1-\delta \leq x\le 1$ satisfies that, in any translation-invariant Gibbs measure, there is either an infinite path or there are infinitely many finite loops surrounding the origin, almost surely. In the intersection of this regime and the proven dichotomy regime, i.e., when $1\le n \leq 1+\delta$, $1-\delta \leq  x\le \tfrac{1}{\sqrt{n}}$, the result implies the existence of macroscopic loops in finite domains and the associated RSW theory.

This is the first result to show that macroscopic loops occur on a regime of parameters with positive Lebesgue measure.
    \item[\emph{(ii)}] In the regime $n= 1, x\in (1, \sqrt{3}]$, which corresponds to an antiferromagnetic Ising model, in any Gibbs measure, there is either an infinite path or there are infinitely many finite loops surrounding the origin, almost surely.
    \item[\emph{(iii)}] In the parameter range $1\le n\le 2$, $x=1$, it is shown that the model on a torus exhibits a non-contractible loop with uniformly positive probability.
\end{itemize}
More precise statements as well as additional finite-volume consequences will appear in the three subsections below.

The new technique is based on a `XOR trick'. The XOR trick is straightforward in the case of critical percolation (see Section \ref{section:XOR}). Its application to other values of $n$ and $x$ is non-obvious; our analysis uses expansions in $n$ and $x$ and requires delicate control of the percolative properties of the resulting graphical representations. 
This control, in turn, relies on recent bounds on the site percolation threshold of Benjamini--Schramm limits of finite planar graphs. 

Previous techniques for showing the existence of long loops relied on positive association (FKG) properties for a suitable spin representation; such representations are only known to exist in the regimes $n\ge 1, x\le\tfrac{1}{\sqrt{n}}$~\cite{DumGlaPel21} and $n\ge 2, x\le \tfrac{1}{\sqrt{n-1}}$~\cite{GlaMan21}. Among the merits of the new technique is that it applies in the absence of such FKG properties. A second merit is that the technique makes little use of the specific structure of the underlying hexagonal lattice; the conclusion of~Theorem~\ref{thm:long-around-1-inf-vol} holds for general quasi-transitive, trivalent planar graphs (such as the `bathroom tiling' graph).
Moreover, only mild use is made of quasi-transitivity (for establishing existence of Benjamini--Schramm limits; see Lemma~\ref{lem:uniform-integrability}) and thus the technique might also be of use in studying the loop~$O(n)$ model on non-periodic plane triangulations (e.g., on the Uniform Infinite Planar Triangulation~\cite{AngSch03}).

\smallskip
{\bf Notation.} In this paper, we embed the hexagonal lattice $\bbH$ and its dual triangular lattice~$\bbT$ in the (complex) plane so that vertices of~$\bbT$ (centers of faces of~$\bbH$) are identified with numbers~$k+\ell e^{i\pi/3}$ where~$k,\ell\in\bbZ$. The face of~$\bbH$ centered at the origin is denoted by~$\zero$.

For a non-negative integer~$k$, and some face~$f$ of~$\bbH$, let~$\Lambda_k(f)$ be the subgraph of~$\bbH$ induced by all vertices bordering faces in the graph ball of radius~$k$ around~$f$ in~$\bbT$. 
Below, with a slight abuse of notation, we refer to~$\Lambda_k(f)$ as the ball of radius~$k$ around~$f$.
We implicitly assume that~$\Lambda_k$ (with no face~$f$ indicated) is the ball around~$\zero$.

We say that a subgraph~$\calD$ of~$\bbH$ is a {\em domain} if it consists of all vertices and edges surrounded by some self-avoiding cycle on~$\bbH$ (including the cycle itself). In particular, the balls defined above are domains.

For a general graph~$G$, we denote the set of its edges by~$E(G)$. 
Set $\dist_{G}$ to be the graph distance in $G$, and define $B_r(v)$ to be the combinatorial ball of radius $r$ around a vertex $v$ in $G$ (with respect to this metric).

\subsection{Results in the regime $1\le n\le 1+\delta$ and $1-\delta\le x\le 1$}
\label{sec:res-near-1-1}

It is convenient to first state our results in terms of infinite-volume measures and then pass to their finite-volume consequences. We will use the DLR formalism (due to Dobrushin~\cite{Dob68} and Lanford and Ruelle~\cite{LanRue69}). A measure $\bbP_{n,x}$ is an \emph{infinite-volume Gibbs measure} of the loop $O(n)$ model with edge-weight $x$, if $\bbP_{n,x}$ is supported on loop configurations (i.e.,  configurations of loops and bi-infinite paths) and satisfies the following property. Let $\omega$ be a sample from $\bbP_{n,x}$. Then, for any finite subgraph $\calD$ of $\bbH$, conditioning on the restriction of $\omega$ to $\calD^c$, almost surely, the distribution of $\omega$ is given by $\bbP_{\calD,n,x}^\omega$ (noting that this measure is determined by the restriction of $\omega$ to $\calD^c$). The measure $\bbP_{n,x}$ is called \emph{translation-invariant} if it is invariant under all translations of the plane preserving the lattice $\bbH$.

\begin{thm}\label{thm:long-around-1-inf-vol}
  There exists $\delta>0$ such that the following holds. Let $\bbP_{n,x}$ be a translation-invariant Gibbs measure of the loop $O(n)$ model with
  \begin{equation}\label{eq:theorem-1-regime}
    1\le n\le 1+\delta,\quad 1-\delta\le x\le 1.
  \end{equation} Then,
  \begin{equation*}
  	\bbP_{n,x}\,\big(\{\text{$\exists$ bi-infinite path }\} \cup \{ \text{every face is surrounded by  infinitely many loops} \}\big)=1.
  \end{equation*}
\end{thm}

To place the theorem in context, we briefly discuss some of the beliefs regarding the number and structure of infinite-volume Gibbs measures. When~$0<n\le 2$ and $0<x<\infty$, the loop $O(n)$ model is expected to have a unique Gibbs measure  (when both $n$ and $nx^6$ are large, the model has multiple periodic Gibbs measures; see~\cite{DumPelSamSpi17}). 
In addition, for any~$n>0$ and $0<x<\infty$, all Gibbs measures are expected to be periodic and supported on loop configurations without bi-infinite paths. If these statements were established (for the regime~\eqref{eq:theorem-1-regime}), the theorem would imply that the unique Gibbs measure has infinitely many loops surrounding every vertex, almost surely. 
However, unicity of the Gibbs measure and the absence of bi-infinite paths are currently only proven when the model has a positively associated representation: when~$n\geq 1,x\leq \tfrac{1}{\sqrt{n}}$~\cite{DumGlaPel21,GlaMan21c} and at~$n=2,x=1$~\cite{GlaMan21}.

The theorem rules out the possibility of exponential decay for the loops in the regime~\eqref{eq:theorem-1-regime} (see Corollary~\ref{cor:long-around-1} below). In the intersection with the dichotomy regime $n \geq 1, x \leq \tfrac{1}{\sqrt{n}}$, one thus concludes the existence of macroscopic loops in finite domains and the associated RSW theory. For instance, the following finite-volume statement is an immediate corollary of~\cite[Theorem 1]{DumGlaPel21}.

\begin{cor}\label{cor:RSW}
There exist constants $\delta,c >0$ for which the following holds. Let $1\le n\le 1+\delta$ and $1-\delta\le x\le \tfrac{1}{\sqrt{n}}$. For any $k>2$ and any loop configuration $\xi$,
\begin{equation}\label{eq:RSW}
c\le \bbP_{A_k,n,x}^\xi[\exists\text{ a loop in $A_k$ surrounding } \zero]\le 1-c,
\end{equation}
where $A_k := \Lambda_{2k}(\zero) \setminus \Lambda_k(\zero)$ is an annulus with inner radius $k$ and outer radius $2k$.
\end{cor}

We proceed to elaborate on the finite-volume consequences of Theorem~\ref{thm:long-around-1-inf-vol} in the full regime~\eqref{eq:theorem-1-regime}. The first issue to address is finding a sequence of domains and boundary conditions for which the loop $O(n)$ model converges to a translation-invariant Gibbs measure in the thermodynamic (subsequential) limit. We use the natural choice of taking the domains to be balls of growing radius (more generally, F{\o}lner sequences) with suitable boundary conditions. To obtain translation-invariance in the limit, the configuration is considered from the point of view of a uniformly chosen face $\mathbf{f}$. 

A technical point which must be addressed is that it is {\em a priori} unclear that thermodynamic limits are Gibbs measures. 
This issue arises for infinite-volume limit measures which give positive probability to the event of multiple bi-infinite paths.
Two distinct bi-infinite paths can arise as limits of one or two loops in the finite-volume measures~--- and this affects the probability distribution on local rewirings of a finite-volume configuration.
Hence, the limiting distribution on these rewirings will include a non-trivial term corresponding to the way the bi-infinite paths are wired.
This violates the DLR condition formulated above Theorem~\ref{thm:long-around-1-inf-vol} (see, e.g., \cite{haggstrom1996random} for more details and discussion).
%The issue may be in the wiring of interfaces outside of a finite volume; if there are multiple bi-infinite paths, rewiring a configuration in finite volume could change the number of loops inside, thus violating the DLR condition in the thermodynamic limit (see, e.g., \cite{haggstrom1996random}).
However, in our context of translation-invariant limit measures, the Gibbs property may be derived from a version of the Burton--Keane argument (see Proposition \ref{prop:gibbs}). Therefore, Theorem~\ref{thm:long-around-1-inf-vol} implies that, as the domains grow, one either has a long loop (converging to the bi-infinite path) in the vicinity of $\mathbf{f}$, or else one has longer and longer loops surrounding~$\mathbf{f}$.

We may further conclude that the length of the loop passing near a typical face $\mathbf{f}$ is not a uniformly integrable sequence of random variables. To give a precise statement, we begin with some definitions.
For a loop configuration $\omega$ and a face $f$ of the hexagonal lattice, we set $\calL(f)$ to be the length of the longest loop that borders $f$ in $\omega$ (setting it to zero if no such loop exists and to infinity if~$f$ borders a bi-infinite path). We set $C_{r,R}(f)$ to be the event that there exists a loop or a bi-infinite path in $\omega$ which intersects both $\Lambda_r(f)$ and $\Lambda_R(f)^c$; we also set $S_{r,R}(f)$ to be the event that there exists a loop in $\omega$ which is contained in $\Lambda_R(f) \setminus \Lambda_r(f)$ and surrounds $f$. A sequence of random variables $\{X_k\}$ is called {\em uniformly integrable} if
\begin{equation}\label{eq:ui-def}
	\lim_{r\to \infty} \sup_{k} \mathbb{E}[|X_k| \cdot 1_{|X_k| \geq r}] = 0.
\end{equation}
In particular, if the sequence $\{X_k\}$ is not uniformly integrable then $\sup_k\mathbb{E}|X_k|^{1+\eps}=\infty$ for all $\eps>0$.

\begin{cor}\label{cor:long-around-1}
There exists a constant $\delta >0$ for which the following holds. Assume that \eqref{eq:theorem-1-regime} holds. Let $\{\xi_k\}$ be loop configurations, $\omega_k$ be sampled from $\bbP_{\Lambda_k(\zero),n,x}^{\xi_k}$, and $\mathbf{f}_k$ be a uniformly chosen face of $\Lambda_k(\zero)$, sampled independently of $\omega_k$. Then
\[
\lim_{r \to \infty} \lim_{R \to \infty} \liminf_{k \to \infty} \mathbb{P}\left[C_{r,R}(\mathbf{f}_k) \cup S_{r,R}(\mathbf{f}_k)\right] = 1.
\]
In particular, the sequence of random variables $\{\calL(\mathbf{f}_k)\}$ is not uniformly integrable.
\end{cor}

\subsection{Results for $n=1, x\in (1,\sqrt{3}]$ and for the antiferromagnetic Ising model}

In the next theorem, we show that, when~$1<x\leq \sqrt{3}$, the loop~$O(1)$ model exhibits at least one loop whose diameter is comparable to that of the domain. For a loop $\omega \in \bbH$, we define its diameter to be the diameter of its interior, viewed as a subgraph of $\bbT$.

\begin{thm}\label{thm:Ising-antifer}
	There exists a constant $c >0$ such that, for $n=1$ and any~$1<x\leq \sqrt{3}$, any~$k>2$ and any loop configuration~$\xi$ on~$\bbH$ without bi-infinite paths,
	\[
		\bbP_{\Lambda_{2k}(\zero),1,x}^\xi ( \zero \text{ is surrounded by a loop of diameter greater than } k) \geq c.
	\]
\end{thm}

The loops of the loop~$O(1)$ model can be represented as the domain walls of an Ising model defined on the faces of $\bbH$. To be more precise, we define the Ising model: let~$\tau$ be a spin configuration on vertices of~$\bbT$ (or, equivalently, the faces of~$\bbH$), that takes a value~$+1$ or~$-1$ at each vertex. Given a domain~$\calD$ in~$\bbH$, the Ising model with parameter~$\beta$ on the faces of~$\calD$ with boundary conditions~$\tau$ is supported on spin configurations~$\sigma$ that coincide with~$\tau$ outside of~$\calD$ and is given by
\begin{align}\label{eq:def-ising3}
	\mathrm{Ising}_{\calD,\beta}^\tau(\sigma) &= \frac{1}{Z_{\calD,\beta}^\tau}\cdot \exp\left( \beta \cdot \sum_{u\sim v} \sigma(u)\sigma(v) \right),
\end{align}
where~$Z_{\calD,\beta}^\tau$ is a normalising constant and the sum runs over pairs of adjacent faces, at least one of which is in~$\calD$. 

When positive, the parameter $\beta$ should be viewed as the inverse temperature for a ferromagnetic interaction model. In this case, the model is positively associated and satisfies Griffiths' correlation inequalities that greatly aid the analysis; see eg.~\cite{FriVel17} for an introduction. On bipartite graphs, the models at~$\beta$ and~$-\beta$ are equivalent. Here, we focus on the antiferromagnetic case~$\beta<0$ on the triangular lattice, where much less is known.

An Ising configuration $\tau$ can be mapped to a loop configuration $\omega$ by considering the set of edges in $\bbH$ which separate opposite spins. The induced measure on loop configurations is precisely the loop~$O(1)$ measure with $x = e^{-2\beta}$. Thus, the correspondence is such that the Ising model is ferromagnetic when $x < 1$, and antiferromagnetic when $x > 1$. Under this correspondence, we find that the above theorem also has consequences for the antiferromagnetic Ising model on the triangular lattice:

\begin{cor}\label{cor:IsingSpinCircuits}
	Let $\mathrm{Minus}_k$ be the event that there exists a connected component of~$\{v \in \mathbb{T} : \sigma_v = -1\}$ of diameter at least $k$ that either contains $\zero$ or surrounds $\zero$. Then, there exists a constant $c>0$ such that, for any $\beta \in [-\tfrac14{\log 3},0)$, any~$k>0$ and any boundary conditions~$\tau$,
	\[
		\mathrm{Ising}_{\Lambda_{2k}(\zero),\beta}^\tau ( \mathrm{Minus}_k ) \geq c.
	\]
\end{cor}

Note that our approach does not produce RSW statements for the antiferromagnetic Ising model. Specifically, we do not prove the surrounding connected component ({\em cluster}) can be found in an annulus of a fixed aspect ratio. We believe this holds in the entire antiferromagnetic regime~$\beta < 0$ but this remains completely open. On the contrary, in the ferromagnetic case this is well-understood: when~$\beta \in [0,\tfrac14{\log 3}]$ the RSW holds (see eg.~\cite{Tas16} for a general strategy) and when~$\beta \in [0,\tfrac14{\log 3}]$ the model is in a strongly ordered phase~\cite{AizBarFer87}.

Our analysis also implies the following result on the Gibbs measures of the loop~$O(1)$ and Ising models.
Unlike in Theorem~\ref{thm:long-around-1-inf-vol}, translation invariance is not assumed.

\begin{cor}\label{cor:IsingGibbs}
	Let $\bbP_{1,x}$ be a Gibbs measure of the loop $O(1)$ model with~$1<x\leq \sqrt{3}$. Then,
  \begin{equation*}
  	\bbP_{1,x}\,\big(\{\text{$\exists$ bi-infinite path}\} \cup \{ \text{every face is surrounded by  infinitely many loops} \}\big)=1.
  \end{equation*}
  Consequently, if~$-\tfrac14\log 3 \leq \beta <0$, then any Gibbs measure~$\mathrm{Ising}_{\beta}$ of the Ising model with parameter~$\beta$ either includes a bi-infinite interface between pluses and minuses, or every face is surrounded by infinitely many finite interfaces.
\end{cor}

\subsection{Results for $n\in [1,2],x=1$ on a torus}

In the next theorem, we show that, when~$n\in [1,2], x=1$, the loop~$O(n)$ model on a torus has a non-contractible loop with uniformly positive probability. Denote by~$\tor_{k}$ a~$k \times k$ torus obtained by identifying the faces on the opposite sides of the parallelogram domain in~$\bbH$ which consists of the faces $s + t e^{i\pi/3}$, where~$0\leq s,t \leq k$ are integers. The measure~$\bbP_{\tor_{k},n,x}$ on loop configurations on~$\tor_{k}$ is defined analogously to the case of planar domains, though no boundary conditions are necessary. We say that a loop is \emph{non-contractible} if it has a non-trivial homotopy when considered as a subset of a continuous torus.

\begin{thm}\label{thm:x-1}
	For any~$n\in [1,2]$ and~$x=1$ and any $k\geq 1$,
	\[
		\bbP_{\tor_{k},n,1} ( \text{there exists a non-contractible loop}) \geq \tfrac14.
	\]
\end{thm}
Note that non-contractible loops are inherently ``long'' in the sense that their length is at least the side-length of the torus. It would be natural to try to extend this statement to planar domains; however, our proof relies on the additional symmetries of the torus.

\subsection{Outline of the main tool: the XOR trick}\label{section:XOR}

The method that is at the heart of the proofs  of this paper is the XOR trick. This trick makes use of the fact that loop configurations form a closed subgroup of $ \{0, 1\}^{E(\bbH)}$, viewed as an Abelian group under component-wise addition modulo $2$. Explicitly, we view a loop configuration $\omega$ as an element of $\{0,1\}^{E(\bbH)}$, where $\omega(e) =1$ if and only if $ e \in \omega$.  For a loop configuration $\omega$ and a simple cycle~$\Gamma$ (that we call {\em circuit} below), define the configuration $\omega \XOR \Gamma$ as the symmetric difference of $\omega$ and~$\Gamma$:
\begin{equation}\label{eq:xor}
	(\omega \XOR \Gamma ) (e) := \omega(e) + \1_{e \in \Gamma}\quad \text{ (mod 2)} \qquad \qquad \forall e \in E(\bbH).
\end{equation}
Each vertex of~$\bbH$ has an even degree in~$\omega$ and~$\Gamma$; therefore, the same is true for~$\omega \XOR \Gamma$, meaning that, given $\Gamma$, the XOR operation is an involution on the set of loop configurations.

The next combinatorial lemma describes how the XOR operation produces large loops.

\begin{lem}\label{lem:xor}
Let $\omega$ be a loop configuration without bi-infinite paths. Then,	for any circuit~$\Gamma$ that surrounds~$\Lambda_k(\zero)$, either~$\omega$ or~$\omega\XOR\Gamma$ has a loop of diameter at least~$k$ that surrounds~$\zero$.
\end{lem}

\begin{proof}
Let~$\omega_\Gamma$ denote the set of loops in~$\omega$ that intersect~$\Gamma$. Since
	\[
		(\omega_\Gamma \XOR \Gamma) \XOR \Gamma  = \omega_\Gamma,
	\]
	all loops in~$\omega_\Gamma \XOR \Gamma$ must intersect~$\Gamma$. Note that
	\[
		\omega \XOR \Gamma  = \Gamma\XOR \omega_\Gamma\XOR (\omega\setminus\omega_\Gamma)= (\Gamma\XOR \omega_\Gamma) \cup (\omega\setminus\omega_\Gamma).
	\]	
	
	If no loop in~$\omega$ surrounding~$\zero$ has diameter at least~$k$, then no loop in $\omega_\Gamma$ surrounds~$\zero$. For any finite loop configuration~$\eta$, there exists a unique~$\sigma_\eta$ assigning~$1$ and~$-1$ to the faces of~$\bbH$ in such a way that:
	\begin{itemize}
		\item for any two adjacent faces~$u$ and~$v$ of~$\bbH$, we have~$\sigma_\eta(u)\neq\sigma_\eta(v)$ if and only if~$\eta$ contains the common edge of~$u$ and~$v$;
		\item the value of~$\sigma_\eta$ on the unique infinite connected component in~$\bbH\setminus \eta$ is~$1$.
	\end{itemize}
	Note that the XOR operation has the following interpretation:
	\[
		\sigma_{\Gamma\XOR \omega_\Gamma}=\sigma_\Gamma \sigma_{\omega_\Gamma}.
	\]
	Since~$\Gamma$ surrounds~$\zero$, but no loop in~$\omega_\Gamma$ surrounds~$\zero$,
	\[
		\sigma_{\Gamma\XOR \omega_\Gamma}(\zero) 
		= \sigma_\Gamma(\zero) \sigma_{\omega_\Gamma}(\zero) = -1.
	\]
	This means that~$\Gamma\XOR \omega_\Gamma$ contains a loop that surrounds~$\zero$.
	The diameter of this loop must be at least $k$ since each loop in~$\Gamma\XOR \omega_\Gamma$ intersects~$\Gamma$.

%	If no loop in~$\omega$ surrounding~$\zero$ has diameter at least~$k$, then no loop in $\omega_\Gamma$ surrounds~$\zero$. Then, since~$\Gamma$ surrounds~$\zero$, the configuration~$\Gamma\XOR \omega_\Gamma$ must contain a loop surrounding~$\zero$ that intersects $\Gamma$.  The diameter of this loop must be at least $k$.
\end{proof}

In order to illustrate the use of Lemma~\ref{lem:xor}, we provide a short `folkloric' proof that, in the $O(1)$ loop model with $x=1$ (which corresponds to critical percolation on $\bbT$), the origin is surrounded by arbitrarily large loops, almost surely. This conclusion is not new; in fact, it is a weaker version of the classical Russo--Seymour--Welsh (RSW) estimates for critical percolation \cite{Rus78,SeyWel78}. Nevertheless, the main point here is that the XOR trick is actually more robust than one might think and is applicable in settings where relatively little is understood about the underlying probability measure on loop configurations. A heuristic description of the argument can be seen in Figure \ref{fig:PercXor}.

\begin{prop}\label{prop:perco}
	For any~$k$ and any loop configuration~$\xi$ with no bi-infinite paths, one has
	\[
		\bbP_{\Lambda_k(\zero) ,1,1}^\xi \, \big( \text{there exists a loop of diameter $\geq k$ surrounding $\zero$}\big) \geq \tfrac{1}{2}.
	\]
\end{prop}
The limsup of these events is equal to the event that there are infinitely many loops surrounding $\zero$. Since this limsup is a translation-invariant event, ergodicity, Fatou's lemma, and the above proposition imply that the loop~$O(1)$ model with $x=1$ has infinitely many loops surrounding the origin, almost surely. 
\begin{proof}
	Let~$A_k$ be the event that there exists a loop of diameter $\geq k$ surrounding $\zero$. Set~$\Gamma$ to be the boundary of~$\Lambda_k(\zero)$. Since the~ XOR-operation is measure-preserving, Lemma~\ref{lem:xor}, together with the union bound, implies that
	\begin{align*}
		2\cdot \bbP_{\calD,1,1}^\xi(\omega\in A_k) &= \bbP_{\calD,1,1}^\xi(\omega\in A_k) + \bbP_{\calD,1,1}^\xi(\omega \XOR \Gamma\in A_k) \\ & \geq \bbP_{\calD,1,1}^\xi(\omega \in A_k \text{ or } (\omega \XOR \Gamma) \in A_k) = 1.
	\end{align*}
\end{proof}

\begin{figure}
\begin{center}
	\includegraphics[width=0.27\textwidth]{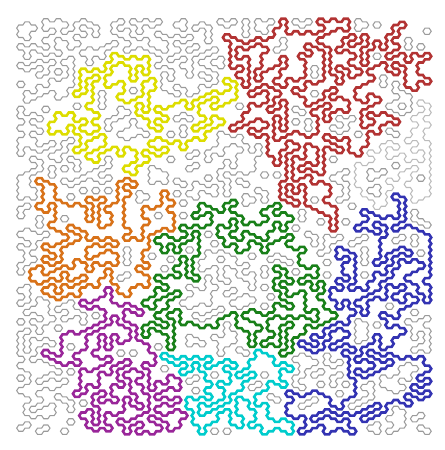}\quad
	\includegraphics[width=0.27\textwidth]{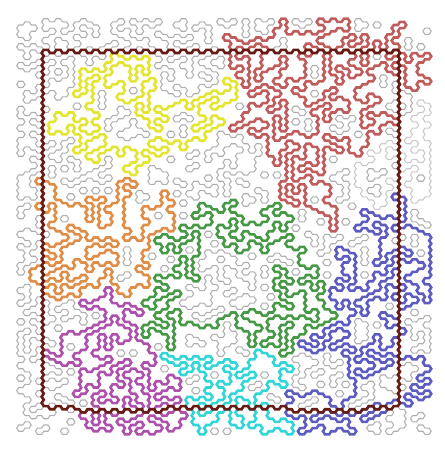}\quad
	\includegraphics[width=0.27\textwidth]{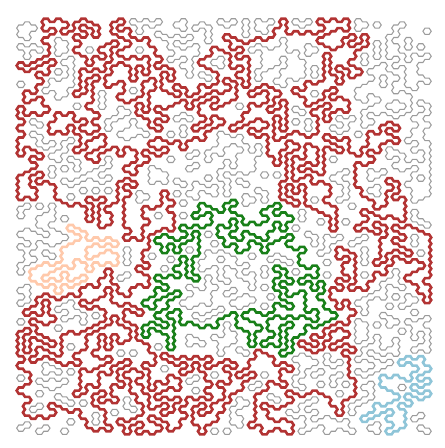}
	\end{center}
	\caption{\emph{Left:} a sample of of the loop $O(1)$ model with $x=1$ and empty boundary conditions, with the longest loop in red. \emph{Middle:} a large loop superimposed on the configuration. \emph{Right:} the result of the XOR operation between the configuration and the loop. Again, the longest loop is in red. Note that parts of the orange and blue loops are cut into smaller loops by the operation.}\label{fig:PercXor}
\end{figure}

Compared to this `toy' case~$n=x=1$, the difficulty in the proofs of our main results lies in the fact that, as soon as~$n > 1$ or~$x < 1$, the measure~$\bbP_{\calD,n,x}^\xi$ is not uniform, and hence the XOR operation is not measure-preserving. In order to surmount this difficulty, we consider an expansion around~$n=x=1$ in both the $n$ and $x$ variables. We say that any loop of~$\omega$ is a {\em defect loop} with probability $(n-1)/n$, and that any edge in the complement of~$\omega$ is a {\em defect edge} with probability $1-x$; the process is defined independently for each loop and edge. As we will show below, given a defect subgraph $\eta$, the distribution of $\omega \setminus \eta$ is uniform on the set of loop configurations on the complement of $\eta$. In particular, $\omega\setminus \eta$ is invariant under the~XOR operation with any circuit~$\Gamma$ which is disjoint from~$\eta$. Therefore, our goal is to show that these defect subgraphs allow for a large vacant circuit. We work with defect edges to prove Theorem~\ref{thm:Ising-antifer}, defect loops to prove Theorem~\ref{thm:x-1}, and both defect edges and loops to prove Theorem~\ref{thm:long-around-1-inf-vol}; we present the theorems in this order, building up the complexity as the paper progresses.   

Development in~$n$ appears in Chayes--Machta~\cite{ChaMac98}, whereas expansion in~$x$ corresponds to the classical Edwards--Sokal coupling~\cite{EdwSok88} and its generalization by Newman~\cite{New90,New94} to the antiferromagnetic case. As far as we are aware, this paper is the first usage of both expansions simultaneously.  

\subsection*{Acknowledgements}
The authors would like to thank Michael Krivelevich for sharing knowledge about planar graphs, and Ioan Manolescu and Yinon Spinka for fruitful discussion of the work presented here and help with the figures.

AG, MH, and RP were supported in part by the Israel Science Foundation grant 861/15 and the European Research Council starting grant 678520 (LocalOrder). NC was supported by the Israel Science Foundation grant number 1692/17. AG was supported by the Swiss NSF grants P300P2\_177848, ~P3P3P2\_177850, P2GEP\_165093 and by the Austrian Science Fund (FWF) 10.55776/P34713. MH was supported by the Zuckerman Postdoctoral Fellowship. RP and MH were further supported by the Israel Science Foundation grant 1971/19.

\section{$n=1, x\in (1,\sqrt{3}]$}

The aim of this section is to prove Theorem~\ref{thm:Ising-antifer}. The antiferromagnetic Ising model, which corresponds to the loop~O(1) model with $x >1$, can be represented as a conditioned FK-Ising model using Newman's extension~\cite{New90,New94} of the classical Edwards--Sokal coupling. Known results on the standard FK-Ising model will then imply that the XOR trick is applicable for the loop~$O(1)$ model with~$x\in (1,\sqrt{3}]$, thus proving Theorem~\ref{thm:Ising-antifer}.

We fix~$n=1$ until the end of this section.

\subsection{Link to the Ising model}\label{sec:ising}

The loop~$O(1)$ model on a domain $\calD$ of~$\bbH$ can be viewed as a domain wall, or low-temperature, representation of the Ising model on the faces of the lattice; see for example~\cite[Section 3.7.2]{FriVel17}. The next lemma states this connection explicitly.

For any spin configuration $\tau: \bbT \rightarrow \{-1,1\}$ and a domain $\calD$, we set $\calS(\calD,\tau)$ to be the set of spin configurations that match $\tau$ on the faces of $\calD^c$. Let $\mathrm{DW}(\tau)$ be the subgraph of $\bbH$ where $e \in \mathrm{DW}(\tau)$ if and only if the edge borders on two faces with different values of $\tau$.

\begin{lem}\label{lem:ising-o1}
	Let~$\tau: \bbT \rightarrow \{-1,1\}$ be a spin configuration, $\calD$ be a domain, and $\xi = \mathrm{DW}(\tau)$. Then, $\mathrm{DW}$ is a bijection from $\calS(\calD,\tau)$ to $\calE(\calD,\xi)$. Moreover, if $\sigma \sim \mathrm{Ising}_{\calD,\beta}^{\tau}$, then $\mathrm{DW}(\sigma) \sim  \bbP_{\calD, 1,x}^{\xi}$, where $\beta = -\tfrac12 \log x$. 
\end{lem}

\begin{proof}
	Since the faces of $\calD^c$ form a connected subset of~$\bbT$, it is standard that~$\mathrm{DW}$ is a bijection.
	Let $\sigma \in \calS(\calD,\tau)$ and set $\omega = \mathrm{DW}(\sigma)$. The probability of $\omega$ under $\bbP_{\calD, 1,x}^\xi$ can be written in the following way:
	\begin{align*}	
	\label{eq:colour_measure}
	\bbP_{\calD, 1,x}^\xi(\omega) &= \frac{1}{Z_{\calD,1,x}^\xi}\cdot x^{|\omega|} =  \frac{1}{Z_{\calD,1,x}^\xi}\cdot x^{\#\{u\sim v \colon \sigma(u) \neq \sigma(v)\}} \\
	&=\frac{1}{Z_{\calD,1,x}^\xi}\cdot \exp\Big( \log x \cdot  \sum_{u\sim v} \1_{\{\sigma(u) \neq \sigma(v)\}}\Big)\\
	&= \frac{1}{Z_{\calD,\beta}^{\tau}} \exp\Big( \beta\sum_{u\sim v} \sigma(u)\sigma(v)\Big) = \mathrm{Ising}_{\calD,\beta}^\tau(\sigma),
\end{align*}
where all sums are taken over edges $(u,v)$ with at least one vertex in $\calD$, and we used that~$\1_{\{\sigma(u) \neq \sigma(v)\}} =\tfrac12 - \tfrac12\sigma(u)\sigma(v) $ and~$\beta = -\tfrac12 \log x$.
\end{proof}

Observe that $x \in (0,1)$ implies~$\beta >0$ and thus the law of $\sigma$ is that of a ferromagnetic Ising model on the faces of $\calD$. The case $x = 1$ corresponds to the uniform measure on all $\pm 1$ spin configurations on the faces of $\calD$ -- that is, to Bernoulli site percolation with parameter~$1/2$ on the faces of $\calD$. Finally, $x > 1$ implies~$\beta <0$ and thus~$\sigma$ has the law of an antiferromagnetic Ising model.

As was first shown by Onsager~\cite{Ons44}, the value $\beta = \frac{\log 3}4 $, or equivalently~$x = \frac{1}{\sqrt3}$, is the critical point for the ferromagnetic Ising model on the triangular lattice  (see~\cite{BefDum12} for the explicit formula on the triangular lattice). Sharpness of the phase transition was established in~\cite{AizBarFer87} and implies that, when~$\beta > \frac{\log 3}4$ (that is~$x < \frac{1}{\sqrt3}$), the model is in an ordered phase, with exponentially small pockets of $-1$ in an environment of $+1$.

At the same time, when $0  \leq \beta \leq \frac{\log 3}4$ (that is $\frac{1}{\sqrt3} \leq x \leq 1$), the model is in a disordered phase, with macroscopic clusters of both $+1$ and $-1$; this is a consequence of  the proof in~\cite{Tas16}.

\subsection{FK-Ising representation in the antiferromagnetic regime}
\label{sec:FK-Ising}

In this section, our aim is to describe an FK-type representation of the loop~$O(1)$ model.

Given a domain~$\calD$, define the dual domain~$\calD^*$ as the subgraph of the triangular lattice~$\bbT$ induced by the set of vertices that correspond to faces bordering at least one edge of $\calD$.

\subsubsection{Ferromagnetic Ising model: $0<x<1$}
Fix a spin configuration $\tau$ and assume $0 <x <1$. We consider a joint law of spin configurations $\sigma \in \calS(\calD,\tau)$ and subgraphs $\eta$ of $\calD^*$ whose marginal on $\sigma$ is precisely the Ising measure. To do so, consider the following expansion: 
\begin{align*}
	\mathrm{Ising}_{\calD,\beta}^\tau(\sigma)
	=& \frac{1}{Z_{\calD,1,x}^\xi}\cdot x^{\#\{u\sim v \colon \sigma(u) \neq \sigma(v)\}} = \frac{x^{|E(\calD)|}}{Z_{\calD,1,x}^\xi}\cdot\left(\tfrac{1}{x}\right)^{\#\{u\sim v \colon \sigma(u) = \sigma(v)\}} \\
	&= \frac{x^{|E(\calD)|}}{Z_{\calD,1,x}^\xi} \cdot \prod_{u\sim v} (1 + (\tfrac{1}{x}-1)\1_{\sigma(u)= \sigma(v)}) \\
	&= \frac{x^{|E(\calD)|}}{Z_{\calD,1,x}^\xi} \cdot \sum_{\eta\subset E(\calD^*)} (\tfrac{1}{x}-1)^{|\eta|}\prod_{uv\in \eta} \1_{\sigma(u)= \sigma(v)},
\end{align*}
where $\xi = \mathrm{DW}(\tau)$, $\beta = - \tfrac{1}{2} \log x$.

The joint law of~$\sigma$ and~$\eta$ is
\begin{equation}
\label{eq:edwards-sokal-fer}
	\mathrm{ES}_{\calD,x}^\xi(\sigma, \eta) = \frac{x^{|E(\calD)|}}{Z_{\calD,1,x}^\xi} \cdot (\tfrac{1}{x}-1)^{|\eta|} \prod_{uv\in \eta} \1_{\sigma(u)= \sigma(v)}.
\end{equation}
Viewing~$\eta$ as a spanning subgraph of~$\calD^*$, the last term is the indicator that~$\sigma$ has constant value on each cluster of~$\eta$. This is the classical Edwards--Sokal coupling ~\cite{EdwSok88} between the Ising model and FK--Ising model (random-cluster model with~$q=2$). 

A particular case of this coupling is given setting $\tau \equiv +1$, or alternatively $\xi = \emptyset$ (wired boundary conditions). 
We define $\varphi^1_{\calD,x}(\eta)$ to be the marginal of $\mathrm{ES}_{\calD,x}^{\emptyset}$ on $\eta$:
\begin{equation}
\label{eq:FK-Ising-fer}
\varphi_{\calD,x}^1(\eta) =  \frac{x^{|E(\calD)|}}{Z_{\calD,1,x}^\emptyset} \cdot (\tfrac{1}{x}-1)^{|\eta|}2^{k^1(\eta)},
\end{equation}
where~$|\eta|$ denotes the number of edges in~$\eta$ and $k^1(\eta)$ is the number of clusters in $\eta$ when all boundary vertices of~$\calD^*$ (that is, faces outside of $\calD$) are identified.

\subsubsection{Antiferromagnetic Ising model: $x>1$}

For this subsection, we assume $x >1$. Following the same steps as for the previous case (as was done by Newman~\cite{New90,New94}), we find that:
\begin{align*}
\mathrm{Ising}_{\calD,\beta}^\tau(\sigma)
	=& \frac{1}{Z_{\calD,1,x}^\xi}\cdot x^{\#\{u\sim v \colon \sigma(u) \neq \sigma(v)\}} = \frac{1}{Z_{\calD,1,x}^\xi} \cdot \prod_{u\sim v} (1 + (x-1)\1_{\sigma(u)\neq \sigma(v)})\\
	=& \frac{1}{Z_{\calD,1,x}^\xi} \cdot \sum_{\eta\subset E(\calD^*)} (x-1)^{|\eta|}\prod_{uv\in \eta} \1_{\sigma(u)\neq \sigma(v)}.
\end{align*}
We again consider the joint law of~$\sigma$ and~$\eta$, which takes a form similar to~\eqref{eq:edwards-sokal-fer}:
%Writing  as coupling of~$\omega$ and~$\eta$, this  in terms of the loop~$O(1)$ model, Edwards--Sokal coupling takes form:
\begin{equation}
\label{eq:edwards-sokal-antifer}
	\mathrm{ES}_{\calD,x}^\xi(\sigma, \eta) = \frac{1}{Z_{\calD,1,x}^\xi} \cdot (x-1)^{|\eta|} \prod_{uv\in \eta} \1_{\sigma(u)\neq \sigma(v)}.
\end{equation}
Viewing~$\eta$ as a spanning subgraph of~$\calD^*$, the last term is the indicator that~$\sigma \in \calS(\calD,\tau)$ defines a proper coloring of~$\eta$ in~$+1$ and~$-1$, i.e. any two vertices linked by an edge of~$\eta$ have {\em different} spins in~$\sigma$. Let $\mathrm{Bip}_\xi(\eta)$ denote the event that there exists such a proper coloring.

The marginal distribution on $\eta$ for $x >1$ is given by
\begin{equation}
\label{eq:FK-Ising-antifer}
	\varphi_{\calD,x}^\xi(\eta) = \frac{1}{Z_{\calD,1,x}^\xi}\cdot  (x-1)^{|\eta|}2^{k^1(\eta)} \cdot \1_{\mathrm{Bip}_\xi(\eta)}.
\end{equation}
In particular, we see that, for any $x > 1$, we have the relation
\begin{equation}\label{eq:Ferro-AntiferroRelation}
\varphi_{\calD,x}^\xi(\eta) = \varphi_{\calD,1/x}^1(\eta \, \mid \mathrm{Bip}_\xi(\eta)).
\end{equation}

\subsection{Input from the FK-Ising model}

For~$x<1$, $\varphi_{\calD,x}^{1}$ is the standard wired FK-Ising measure. The classical result of Aizenman, Barsky, and Fernandez~\cite{AizBarFer87} proves that the FK-Ising model undergoes a sharp phase transition at the critical point (see~\cite{DumTas15} for a recent short proof). The exact value of the critical point on the triangular lattice is~$x=\tfrac{1}{\sqrt{3}}$, as was first shown by Onsager~\cite{Ons44} (see also~\cite{BefDum12} in which the triangular lattice is addressed explicitly). Together, these results imply that, for any~$x\in (\tfrac{1}{\sqrt{3}}, 1)$, the model exhibits exponential decay -- that is, there exists~$c'=c'(x) >0$ such that, for any~$k$ and any domain $\calD$ containing $\Lambda_k(\zero)$,
\begin{equation}\label{eq:fk-ising-exp-decay}
	\varphi_{\calD,x}^1(\zero \text{ is connected to distance } k) <e^{-c'k}.
\end{equation}

The RSW theory developed in~\cite{DumHonNol11} for the FK-Ising model implies that, at the critical point~$x=\tfrac{1}{\sqrt{3}}$, the connection probability decays as a power law in the distance and the probability to cross an annulus~$\Lambda_{2k}(\zero) \setminus \Lambda_k(\zero)$ is bounded above and below uniformly in~$k$. We will only require the following upper bounds:

\begin{prop}\label{prop:fk-ising-input}
	Let~$x\in [\tfrac{1}{\sqrt{3}}, 1)$. Then, there exists a constant~$c>0$ such that, for any integer~$k>2$ and domain~$\calD \supset\Lambda_{2k}(\zero)$,
	\begin{equation}\label{eq:ising-crossing}
		\varphi_{\calD,x}^1(\exists \text{ crossing from } \Lambda_{k}(\zero) \text{ to } \bbT\setminus \Lambda_{2k}(\zero)) < 1-c.
	\end{equation}
\end{prop}

We note that, for~$x\in (\tfrac{1}{\sqrt{3}}, 1)$, the proposition follows from~\eqref{eq:fk-ising-exp-decay}, the union bound, and the finite energy property.

It is also known that, for any~$0<x<1$, the FK-Ising measure~$\varphi^1_{\calD,x}$ is positively associated (see, e.g.,~\cite[Theorem 3.8]{Gri06}). Specifically, define a partial order on~$\{0,1\}^{E(\calD^*)}$ by saying that~$\eta\preceq\eta'$ if~$\eta(e)\leq \eta'(e)$ for all~$e\in E(\calD^*)$. We say that an event~$A\subset \{0,1\}^{E(\calD^*)}$ is increasing  (resp. decreasing) if for any~$\eta\in A$, $\eta\preceq \eta'$ (resp. $\eta' \preceq \eta$) implies~$\eta'\in A$.

\begin{lem}[Positive association]\label{lem:fk-ising-fkg}
	For any two increasing events~$A,B\subset \{0,1\}^{E(\calD^*)}$ of positive probability, we have
	\begin{equation}
		\varphi_{\calD,x}^1(A\, | \, B) \geq \varphi_{\calD,x}^1(A).
	\end{equation}
\end{lem}
Is is straightfoward to confirm that, for any loop configuration $\xi$, the event $\mathrm{Bip}_\xi$ is decreasing in $\eta$. Combining this with \eqref{eq:Ferro-AntiferroRelation}, we conclude that, for any {\em decreasing} event $D$, any $x >1$, and any loop configuration $\xi$,
\begin{equation}\label{eq:CompForDecEvents}
\varphi_{\calD,x}^\xi(D) \geq \varphi_{\calD,1/x}^1(D).
\end{equation}

\subsection{Proof of Theorem \ref{thm:Ising-antifer}, Corollary \ref{cor:IsingSpinCircuits}, and Corollary \ref{cor:IsingGibbs}}
\label{sec:thm-antifer}

\begin{proof}[Proof of Theorem \ref{thm:Ising-antifer}]
Let~$x\in (1,\sqrt{3}]$, fix $k>2$ to be an integer and set $\calD = \Lambda_{2k}(\zero)$. Consider the Edwards--Sokal measure $\mathrm{ES}^{\xi}_{\calD,x}$. For any circuit in $\calD$ and a subgraph $\eta$ of $\calD^*$, we say that the circuit crosses $\eta$ if one of its constituent edges is associated with an edge in $\eta$ (recalling that there is a bijection between the edges of $\calD$ and $\calD^*$). Define~$\mathrm{Circ}_k$ to be the event that there exists a circuit~$C$ of~$\bbH$ entirely contained in the annulus $\Lambda_{2k}(\zero)\setminus \Lambda_k(\zero)$ which does not cross any edge of~$\eta$. By planar duality, the event~$\mathrm{Circ}_k$ is complementary to the existence of a crossing in~$\eta$ from~$\Lambda_{k}(\zero)$ to~$\bbT\setminus \Lambda_{2k}(\zero)$. By observing that $\mathrm{Circ}_k$ is a decreasing event in $\eta$, we can use Proposition~\ref{prop:fk-ising-input} and \eqref{eq:CompForDecEvents} to conclude that
\begin{equation}\label{eq:ising-circuit}
	\varphi_{\calD,x}^{\xi}(\mathrm{Circ}_k) \geq  \varphi_{\calD,1/x}^1(\mathrm{Circ}_k)  = 1 - \varphi_{\calD,1/x}^1(\exists \text{ crossing from } \Lambda_{k}(\zero) \text{ to } \bbT\setminus \Lambda_{2k}(\zero)) > c.
\end{equation}

Sample~$\eta$ from~$\varphi_{\calD,x}^{\xi}$.
Assume~$\eta\in\mathrm{Circ}_k$ and let~$C\subset E(\bbH)$ be the outermost circuit witnessing~$\mathrm{Circ}_k$. Let $\sigma \in \calS(\calD,\tau)$, and set $\omega = \mathrm{DW}(\sigma)$. We define $\sigma \XOR C$ to be the unique spin configuration in $\calS(\calD,\tau)$ such that $\mathrm{DW}(\sigma \XOR C) = \omega \XOR C$; equivalently, $\sigma \XOR C$ is equal to $\sigma$ on the faces outside $C$, and equal to $-\sigma$ on the faces in the interior of $C$. Thanks to \eqref{eq:edwards-sokal-antifer} and the construction of $C$, we have that
\[
\mathrm{ES}_{\calD,x}^\xi(\sigma, \eta) = \mathrm{ES}_{\calD,x}^\xi(\sigma \XOR C, \eta).
\]
Let $S_k$ be the event that $\mathrm{DW}(\sigma)$ contains a loop of diameter at least~$k$ surrounding $\zero$. 
By Lemma~\ref{lem:xor} and since~$\xi$ contains no bi-infinite paths, either~$\omega$ or~$\omega \XOR C$ must be in $S_k$, and the union bound implies that
\begin{align*}
 \mathrm{ES}_{\calD,x}^\xi( \sigma  \in S_k \mid  \mathrm{Circ}_k) \geq \tfrac{1}{2}.
\end{align*}
By definition of the Edwards--Sokal coupling and Lemma~\ref{lem:ising-o1}, the probability that the loop $O(1)$ model contains a loop of diameter greater than $k$ which surrounds the origin is precisely $\mathrm{ES}_{\calD,x}^\xi( \sigma \in S_k)$. From above, we see that 
\[
\mathrm{ES}_{\calD,x}^\xi( \sigma \in S_k) \geq \mathrm{ES}_{\calD,x}^\xi( \sigma  \in S_k \mid  \mathrm{Circ}_k) \cdot 	\varphi_{\calD,x}^{\xi}(\mathrm{Circ}_k) \geq c/2,
\]
completing the proof of the theorem.
\end{proof}

\begin{proof}[Proof of Corollary \ref{cor:IsingSpinCircuits}]
Given a spin configuration $\tau$ on the boundary and $k >0$, let $\tau^+_k$ be the spin configuration that is equal to $\tau$ on $\Lambda_{2k+1}(\zero)$, and is identically equal to $+1$ on $\mathbb{H} \setminus \Lambda_{2k+1}(\zero)$. Since the Ising measure only involves nearest neighbor interactions, 
\[
\mathrm{Ising}_{\Lambda_{2k}(\zero),\beta}^\tau =\mathrm{Ising}_{\Lambda_{2k}(\zero),\beta}^{\tau^+_k}.
\]
The domain wall representation of $\tau^+_k$ contains no bi-infinite paths, and thus, by Theorem~\ref{thm:Ising-antifer}, the probability that this domain wall representation contains a loop $\Gamma$ which surrounds $\zero$ and has diameter greater than $k$ is bounded below uniformly. On this event, the sign of $\sigma$ must be constant along all faces that border $\Gamma$ on the interior, and constant on the faces that border $\Gamma$  on the exterior --- and the two signs must be different. Thus, $\sigma$ will be identically $-1$ on exactly one of the two sets, which both have diameter at least $k$ as subgraphs of $\bbT$. The exterior boundary of $\Gamma$ surrounds $\zero$; the interior boundary of $\Gamma$ either surrounds $\zero$ or contains it. In either case, the event $\mathrm{Minus}_k$ occurs, as required. 
\end{proof}

\begin{proof}[Proof of Corollary \ref{cor:IsingGibbs}]
Fix $x \in (1,\sqrt{3}]$, and let $\bbP_{1,x}$ be a Gibbs measure for the loop~$O(1)$ model with parameter $x$. Since every Gibbs measure can be decomposed into an average over extremal Gibbs measures~(see, e.g., ~\cite{Geo11}) it is sufficient to prove that every extremal Gibbs measure $\tilde{\bbP}_{1,x}$ satisfies
\begin{align*}
	&\tilde{\bbP}_{1,x}(\exists \text{ bi-infinite path} ) = 1, \, \text{ or } \, \\
	&\tilde{\bbP}_{1,x}(\text{every face is surrounded by infinitely many loops}) = 1.
\end{align*}

Since the existence of a bi-infinite path is a tail event, we may assume that there are no bi-infinite paths, $\tilde{\bbP}_{1,x}$-almost surely  (otherwise, the probability is $1$ and we are done). Let $A_k$ be the event that $\zero$ is surrounded by a loop of diameter at least $k$. By Theorem~\ref{thm:Ising-antifer}, $\tilde{\bbP}_{1,x}(A_k) > c$ and this implies that the probability that $A_k$ occurs for infinitely many $k$'s is at least $c$. However, when $\{\limsup_k A_k\}$ occurs, there must be infinitely many loops surrounding {\em every} face of $\bbT$; this is a tail event, and thus its probability under $\tilde{\bbP}_{1,x}$ must be one.
\end{proof}

\section{$n\in [1,2], x=1$}
\label{sec:torus}

In this section, we will prove Theorem \ref{thm:x-1} through an expansion of the loop $O(n)$ model in $n$, which was introduced in~\cite{ChaMac98}. A similar expansion was used in~\cite{GlaMan21} and~\cite{GlaMan21b}. Once this expansion is established, a straightforward monotonicity argument, simple geometric and topological considerations on the torus, and the XOR trick will yield the desired lower bound on the probability of finding a non-contractible loop.

Let $\omega_r$ and $\omega_b$ be two loop configurations on $\tor_k$ (red and blue loops). The pair is {\em coherent} if the loops are disjoint. For any $n \geq 1$, define the measure $\rho_{\tor_k,n}$ by
\begin{equation}\label{eq:DefectDistribution2}
\rho_{\tor_k,n}(\omega_r,\omega_b) = \frac{1}{Z_{\sf loop}(\tor_k)}\cdot(n-1)^{\ell(\omega_b)} \cdot \1_{(\omega_r,\omega_b) \text{ is coherent}},
\end{equation}
where~$Z_{\sf loop}(\tor_k)$ is a normalizing constant. As is clear from the definition above, if one conditions on the $\omega_b$, the distribution of $\omega_r$ is the uniform measure on all configurations which are coherent with $\omega_b$. The next proposition shows that $\omega_r \cup \omega_b$ is distributed as a loop $O(n)$ model with edge weight $x=1$.

\begin{prop}\label{prop:bijection-n-only}
For any $n \geq 1$ the measure~$\mathbb{P}_{\tor_k,n,1}$ is equal to the marginal of~$\rho_{\tor_k,n}$ on~$\omega_r \cup \omega_b$.
\end{prop}

\begin{proof}
Let $\omega$ be a loop configuration on $\tor_k$, and let $\calC(\omega)$ be the $2^{\ell(\omega)}$ distinct red/blue colorings of the loops of $\omega$. Each pair $(\omega_r,\omega_b) \in \calC(\omega)$ is coherent, and therefore
\[
\rho_{\tor_k,n}(\omega_r \cup \omega_b = \omega) = \frac{1}{Z_{\sf loop}(\tor_k)}\sum_{(\omega_r,\omega_b) \in \calC(\omega)} (n-1)^{\ell(\omega_b)}  = \frac{1}{Z_{\sf loop}(\tor_k)} \cdot n^{\ell(\omega)},
\]
where the final equality follows from a binomial-type expansion.
\end{proof}

The following lemma states the simple but essential observation that, whenever $n \in [1,2]$, $\omega_r$ stochastically dominates $\omega_b$. In this context, an event $A$ is increasing if, whenever $\omega \preceq \omega'$ (as elements of $\{0,1\}^{E(\bbH)}$) and $\omega \in A$, then $\omega' \in A$. 

\begin{lem}\label{lem:monotonicity}
For any increasing event $A$ and $n \in [1,2]$,
\[
\rho_{\tor_{k},n}[\omega_r \in A] \geq \rho_{\tor_{k},n}[\omega_b \in A] .
\]
\end{lem}
\begin{proof}
If $\omega$ is distributed as $\bbP_{\tor_{k},n,1}$, then the construction above implies that we can sample $(\omega_r,\omega_b)$ conditional on $\omega_r \cup \omega_b = \omega$ by coloring each loop of $\omega$ blue with probability $(n-1)/n$, independently, and coloring the remaining loops red. Thus, conditional on $\omega$, both $\omega_r$ and $\omega_b$ are (complementary) Bernoulli percolations on the loops of $\omega$. Whenever $n \in [1,2]$, $(n-1)/n \leq 1/2$, and thus the red percolation process stochastically dominates the blue percolation process. We conclude that, for any increasing event $A$,
\begin{align*}
\rho_{\tor_{k},n}[\omega_r \in A] & = \rho_{\tor_{k},n}\left[ \rho_{\tor_{k},n}[\omega_r \in A  \, \mid \, \omega_r \cup \omega_b = \omega] \right]
\\&\geq \rho_{\tor_{k},n}\left[ \rho_{\tor_{k},n}[\omega_b \in A  \, \mid \, \omega_r \cup \omega_b = \omega] \right] \\ & = \rho_{\tor_{k},n}[\omega_b \in A].
\end{align*}
\end{proof}

We now state two topological facts about loops $\tor_k$:
\begin{enumerate}[label=\roman*.]
\item If $\Gamma$ is a non-contractible loop in $\tor_k$, then $\tor_k \setminus \Gamma$ is a connected graph, and \label{eq:facti}
\item If $\omega$ is a loop configuration consisting only of contractible loops, then there exists a spin configuration $\sigma_\omega$, assigning $\pm 1$ to the face of $\tor_k$, such that two adjacent faces have different spins if and only if their common edge belongs to $\omega$. \label{eq:factii}
\end{enumerate}
To see both of these facts, we embed $\tor_k$ in the continuous torus, and consider a lift map from the continuous torus to $\mathbb{R}^2$, its universal cover. The image of a loop in $\tor_k$ under a lift is a loop (i.e. a homeomorphic image of the circle) if and only if the loop has trivial homotopy as a subset of the continuous torus. Otherwise, the image will be a homeomorphic image of a segment (as the length of all loops in $\tor_k$ is finite); since a segment does not disconnect $\mathbb{R}^2$ and the lift is bi-continuous, the first fact above follows. For the second fact, consider a single contractible loop $\Gamma$ in $\tor_k$. Its image under a lift will separate the plane into two connected sets. The preimage of the interior under the lift will be the union of a set of faces of $\tor_k$ (for an explicit proof of this, see~\cite[Theorem 3.2]{feldheim2018rigidity}). We denote this set of faces by $S_\Gamma$. We can then produce a spin configuration $\sigma_\Gamma$ for the loop by setting $\sigma = +1$ on $S_\Gamma$ and $-1$ otherwise. For a general loop configuration, we can produce a spin configuration by multiplying the spin configurations of the individual loops that comprise $\omega$.
Such spin representation for the loop~$O(n)$ model on a planar domain was introduced in~\cite{DumGlaPel21}; for~$n=1$, this is the classical low-temperature expansion of the Ising model~\cite[Sec. 3.7.2]{FriVel17}).

In the next lemma, we show that, for any configuration $\omega$ of contractible loops, applying the XOR operation with a non-contractible loop $\Gamma$ yields a non-contractible loop.

\begin{lem}\label{lem:XORtoplogy}
	Let $\Gamma$ be a non-contractible loop.  Then, for any loop configuration $\omega$ in $\tor_{k}$, at least one of $\{\omega,\omega \XOR \Gamma\}$ includes a non-contractible loop.
\end{lem}

\begin{proof}
	Assume~$\omega$ and~$\omega \XOR \Gamma$ contain only contractible loops. Let~$\sigma_{\omega}$ and~$\sigma_{\omega\XOR\Gamma}$ be their respective spin representations. Then, the product~$\sigma_{\omega}\sigma_{\omega\XOR\Gamma}$ is a spin representation for the loop configuration made up of only $\Gamma$. Such a spin configuration would have to be constant on the connected components of $\tor_k \setminus \Gamma$; however, $\tor_k \setminus \Gamma$ is connected, leading to a contradiction.
\end{proof}

Let $\Gamma$ be a simple circuit in $\tor_{k}$. Given $\omega_b$, we call $\Gamma$ {\em blue-free} if $\Gamma \cap \omega_b = \emptyset$. Similarly, $\Gamma$ is {\em red-free} if $\Gamma \cap \omega_r = \emptyset$. The duality lemma below shows that one can always find a non-contractible circuit that is blue- or red-free (see~\cite[Lemma 3.6]{GlaMan21} for a planar analogue).

\begin{lem}\label{lem:DualityTorus}
	Any consistent pair of loop configurations $(\omega_r, \omega_b)$ on a torus contains either a blue-free or a red-free non-contractible circuit.
\end{lem}
\begin{proof}
	If the loop configuration contains a non-contractible loop, the statement is trivial. Assume all loops are contractible. Let~$\sigma_b$ be a spin representation for~$\omega_b$. Denote by~$\eta_b$ the edges of $\tor_k \setminus \omega_b$. Let~$\eta_b^*$ be the dual configuration, i.e. the set of pairs of adjacent faces of~$\tor_{k}$ having different spins in~$\sigma_b$. By the standard duality, either~$\eta_b$ or~$\eta_b^*$ contains a non-contractible circuit. If this is~$\eta_b$, then we are done, since such circuit is blue-free. Assume now that~$\eta_b^*$ contains a non-contractible circuit~$\tilde{\Gamma}$. By the definition of~$\eta_b^*$, spins on along~$\tilde{\Gamma}$ alternate. Applying the proof of~\cite[Lemma 3.6]{GlaMan21} to our setup {\em mutatis mutandis} we get that, by local modifications, $\tilde{\Gamma}$ can be made into a red-free non-contractible circuit.
\end{proof}

\begin{proof}[Proof of Theorem \ref{thm:x-1}]
Sample $(\omega_r,\omega_b)$ according to $\rho_{\tor_k,n}$. By Lemma \ref{lem:XORtoplogy}, we may condition on $\omega_b$ and find that
\[
\rho_{\tor_{k},n}[\exists \text{ non-contractible loop in } \omega_r \cup \omega_b \, \mid \, \omega_b] \geq \tfrac12 \cdot \1_{\exists \text{ non-contractible blue-free circuit}}.
\]
By Proposition~\ref{prop:bijection-n-only}, the law of $\omega_r \cup \omega_b$ is $\mathbb{P}_{\tor_k,n,1}$; therefore, we may take expectation over $\omega_b$ and find that
\begin{equation}\label{eq:torus-blue-free-circuit}
	\bbP_{\tor_{k},n,1}[\exists \text{ non-contractible loop in } \omega] \geq \tfrac12  \cdot \rho_{\tor_{k},1} [\exists \text{ non-contractible blue-free circuit}].
\end{equation}
The existence of a blue-free circuit is a decreasing event in $\omega_b$; by Lemma~\ref{lem:monotonicity}, the probability of the RHS of~\eqref{eq:torus-blue-free-circuit} is greater or equal than the same quantity for red-free circuits. By Lemma~\ref{lem:DualityTorus}, the sum of the two probabilities is greater or equal to~$1$. Together, this implies that the RHS of~\eqref{eq:torus-blue-free-circuit} is greater or equal than~$1/4$.
\end{proof}

\section{$n \in [1, 1+\delta], x\in [1-\delta, 1]$}

In this section, we will prove Theorem \ref{thm:long-around-1-inf-vol} and Corollaries \ref{cor:RSW} and \ref{cor:long-around-1}. Here, we will expand the loop $O(n)$ model in both $n$ and $x$ simultaneously, which leads us to `defects' that can be either loops or edges. Instead of working on subsets of $\bbH$ directly, we will generate auxiliary finite planar graphs from an infinite-volume loop configuration, which will converge to an infinite planar graph, in the sense of Benjamini--Schramm convergence. Finally, we will show that the absence of defect-free circuits in the original domains can be coupled to the existence of infinite components in an appropriately chosen Bernoulli percolation process on the Benjamini--Schramm limit; the recent work~\cite{Pel19} will be used to rule out this possibility at a small enough percolation parameter.

\subsection{Defect-loop and -edge representation}
\label{sec:spin-rep}

We begin by extending the notion of coherent loop configurations introduced in Section~\ref{sec:torus}.
Given a domain $\calD$ and a loop configuration $\xi$, we say that loop configurations~$\omega_r,\omega_b$ and an edge configuration $\eta \in \{0,1\}^{E(\calD)}$ are {\em coherent} if they are pairwise disjoint and $\omega_r \cup \omega_b \in \calE(\calD,\xi)$.
We refer to~$\omega_b$ and~$\eta$ as the defect loops and edges, respectively.
In a slight abuse of notation, we say that~$(\omega_r,\omega_b)$ is coherent if~$(\omega_r,\omega_b, \emptyset)$ is coherent.

%This can be thought of as a three-coloring of $\calD$ -- we have red loops in $\omega_r$, blue loops in $\omega_b$, and brown, or ``defect,'' edges in $\eta$.
For any $0 < x \leq 1$ and $n \geq 1$, define the measure $\mu_{\calD,n,x}^\xi$ by
\begin{equation}\label{eq:DefectDistribution}
\mu_{\calD,n,x}^\xi(\omega_r,\omega_b,\eta) = \frac{1}{Z_{\sf def}}\cdot(n-1)^{\ell(\omega_b)}\left(\tfrac{1}{x} -1\right)^{|\eta|} \cdot \1_{(\omega_r,\omega_b,\eta) \text{ is coherent}},
\end{equation}
where~$Z_{\sf def}=Z_{\sf def}^\xi(\calD)$ is a normalizing constant, and $|\eta|$ is the number of edges in $\eta$.

\begin{prop}\label{prop:bijection}
For any $n \geq 1$ and $x \in (0,1]$, the measure~$\mathbb{P}_{\calD,n,x}^\xi$ is equal to the marginal of~$\mu_{\calD,n,x}^\xi$ on~$\omega_r \cup \omega_b$.
\end{prop}

\begin{proof}
Fix a pair of disjoint loop configurations $(\omega_r, \omega_b)$. For brevity, we write~$E$ for~$E(\calD)$ below. 
By using a binomial-style expansion, it is easy to see that
\[
\sum_{\eta \subseteq E\setminus(\omega_r\cup\omega_b)}
%\substack{\eta \in \{0,1\}^{E} \\(\omega_r, \omega_b,\eta) \text{ is coherent}}}
\left(\tfrac{1}{x} - 1 \right)^{|\eta|} = \prod_{e \in E \setminus (\omega_r \cup \omega_b)}\left[1 + \left(\tfrac{1}{x}-1 \right)\right] = x^{-|E \setminus (\omega_r \cup \omega_b)|}.
\]
Therefore, the marginal of $\mu_{\calD,n,x}^{\xi}$ on~$(\omega_r,\omega_b)$ is given by
\[
\mu_{\calD,n,x}^\xi(\omega_r,\omega_b) = \frac{x^{-|E|}}{Z_{\sf def}} \cdot (n-1)^{\ell(\omega_b)} x^{|\omega_r| + |\omega_b|}\1_{(\omega_r,\omega_b) \text{ is coherent}}.
\]
To determine the marginal on~$\omega_r \cup \omega_b$, we fix a loop configuration~$\omega \in \calE(\calD,\xi)$ and consider all possible colorings of its loops in red and blue. Summing over all such possibilities and using~$n-1+1=n$, we get
\[
\mu_{\calD,n,x}^\xi(\omega=\omega_r\cup \omega_b) = \frac{x^{-|E|}}{Z_{\sf def}} \cdot x^{|\omega|}\sum_{\omega_b\cup \omega_r = \omega} (n-1)^{\ell(\omega_b)} = \frac{x^{-|E|}}{Z_{\sf def}} \cdot x^{|\omega|} n^{\ell(\omega)} = \mathbb{P}_{\calD,n,x}^\xi(\omega).
\]
\end{proof}

If $(\omega_r,\omega_b,\eta) \sim \mu_{\calD,n,x}^\xi$, \eqref{eq:DefectDistribution} implies that the distribution of $\omega_r$ conditioned on $(\omega_b,\eta)$ is uniform over all loop configurations such that~$(\omega_r,\omega_b,\eta)$ is coherent.
Suppose $\Gamma$ is a simple circuit in $\calD$. 
We call $\Gamma$ {\em defect-free} for~$(\omega_r,\omega_b,\eta)$ if $\Gamma \cap (\omega_b \cup\eta) = \emptyset$. 
Thus, given~$(\omega_b,\eta)$, the XOR operation with a defect-free~$\Gamma$ is measure preserving on~$\omega_r$:
\[
\mu_{\calD,n,x}^\xi(\omega_r,\omega_b,\eta) = \mu_{\calD,n,x}^\xi(\omega_r \XOR \Gamma,\omega_b,\eta).
\]
Therefore, the XOR trick can be used to control the probability of large loops in terms of the probability that there exists large defect-free circuits.

The next proposition concerns the existence of defect-free circuits in a translation-invariant Gibbs measure~$\bbP_{n,x}$ of the loop~$O(n)$ model when~$n$ and~$x$ are close enough to~$1$.
We first need to introduce the corresponding defect representation.
Given a loop configuration~$\omega$, define~$\mu_{n,x}^{\omega}$ as the measure on triples~$(\omega_r, \omega_b, \eta)$ obtained as follows:
\begin{itemize}
	\item color each loop of $\omega$ blue with probability $\tfrac{n-1}{n}$, independently at random; color the remaining loops red;
	\item let $\eta$ be an independent $(1-x)$-Bernoulli percolation on the edges in $\bbH \setminus \omega$.
\end{itemize}
Define~$\mu_{n,x}$ as the expectation of~$\mu_{n,x}^{\omega}$ with respect to~$\bbP_{n,x}$. 
Note that the distribution of $\omega_r$ conditioned on $(\omega_b,\eta)$ is uniform over all loop configurations such that $(\omega_r,\omega_b,\eta)$ is coherent.
Thus, under~$\mu_{n,x}$ conditioned on~$(\omega_b,\eta)$, the loops in~$\omega_r$ are distributed as domain walls for percolation on the faces of~$\bbH\setminus (\omega_b\cup \eta)$.

\begin{prop}[Existence of defect-free circuits]
\label{prop:PlanarGraphProp}
There exists $\delta >0$ such that the following holds.
Assume that
\[
1 \leq n \leq 1 + \delta, \quad 1- \delta \leq x \leq 1.
\]
Let $\bbP_{n,x}$ be a translation-invariant Gibbs measure for the loop~$O(n)$ model with these parameters, and $\mu_{n,x}$ be its associated defect representation. 
Assume there are no bi-infinite paths $\bbP_{n,x}$-almost surely.
Then, for any positive integer $r$,
\[
	\mu_{n,x}(\exists \text{ a defect-free circuit surrounding } \Lambda_r(\zero)) =1.
\]
\end{prop}

The proof of Proposition~\ref{prop:PlanarGraphProp} relies on a non-trivial lower bound for the site percolation threshold of planar graphs~\cite{Pel19} and is given in Sections~\ref{sec:graph-construction}-\ref{sec:xor-circuit-proof}.
Before diving into the proof of the proposition, we show how to derive Theorem~\ref{thm:long-around-1-inf-vol} from it.

\begin{proof}[Proof of Theorem \ref{thm:long-around-1-inf-vol}, given Proposition \ref{prop:PlanarGraphProp}]
	We may decompose the translation-invariant Gibbs measure $\bbP_{n,x}$ into a convex combination of ergodic measures~\cite[Theorem~14.17]{Geo11}.
	Therefore, we can assume that~$\bbP_{n,x}$ is ergodic and almost surely exhibits no bi-infinite paths.
	In particular, there exists~$R_1\geq 1$ such that
	\[
		\bbP_{n,x}(\exists \text{ a loop intersecting both } \Lambda_r(\zero) \text{ and } \Lambda_{R_1}(\zero)^c) < 1/4.
	\]
	By Proposition \ref{prop:PlanarGraphProp}, we can find~$R_2>R_1$ such that,
	\[
		\mu_{n,x}( \exists \text{ a defect-free circuit in } \Lambda_{R_2}(\zero) \text{ surrounding } \Lambda_{R_1}(\zero))  > 1/2.
	\]	
	Hence, with probability at least $1/4$, there exists a defect-free circuit in $\Lambda_{R_2}(\zero)$ which surrounds $\Lambda_{R_1}(\zero)$ and is disjoint from every loop that intersects $\Lambda_r(\zero)$. 
	Condition on the above event and denote the outermost circuit witnessing it by~$\Gamma$.
	Clearly, $\Gamma$ is~$(\omega_b,\eta)$-measurable.
	By applying the XOR operation to $\Gamma$, we may conclude that
	\[
	\bbP_{n,x}(\exists \text{ a loop surrounding } \Lambda_r(\zero)) > \tfrac{1}{8}.
	\]
	The limsup of the above event as $r$ grows is equal to the existence of infinitely many loops surrounding $\zero$; this event must therefore have probability at least $1/8$. In fact, the probability must be $1$, since the measure was assumed to be ergodic and the event is translation invariant. We complete the proof by noting that, if $\zero$ is surrounded by infinitely many loops, so is every other face.
\end{proof}

\subsection{Percolation formulation of existence of defect-free circuits}
\label{sec:graph-construction}

Throughout the section, we fix a loop configuration~$\omega$ with no bi-infinite paths and a domain~$\calD$ that contains~$\zero$.
All objects constructed in this section will depend on~$\omega, \calD$ but we will omit them from the notation for brevity. 
Additionally, we condition~$\mu_{\calD,n,x}^\omega$ on the event $\{\omega_r\cup \omega_b = \omega\}$. Thus, the remaining randomness involves only the edges of~$\eta$ and the coloring of the loops in~$\omega$.

The goal is to translate the existence of defect-free loops to the language of percolation processes. 
More specifically, we will construct a planar graph~$G$ and a 2-dependent site percolation process $\zeta$ on~$G$ determined by~$(\omega_r,\omega_b,\eta)$.
Heuristically, the construction is such that a defect-free circuit surrounding~$\Lambda_r$ exists, as soon as~$\zeta$ does not percolate.
We then show that this non-percolation probability is at least as large as the probability of the same event under an independent site percolation process. In Section~\ref{sec:benjamini-schramm}, we will approximate this process, defined on a sequence of finite graphs, by a process defined on their Benjamini--Schramm limit.

Let $V$ be the set consisting of loops of $\omega$ that intersect $\calD$, edges of $\calD$ that are disjoint from $\omega$, and faces of $\calD$. Define a projection map $\pi$ from the edges and faces of $\calD$ to $V$ so that every edge belonging to a loop is sent to that loop; $\pi$ is the identity map on all other edges and faces. We define a graph $G$ whose vertex set is $V$, and where two vertices~$u$ and~$v$ of~$G$ are linked by an edge if one of the following conditions is satisfied (see Figure \ref{fig:lattices}):
\begin{enumerate}[label=(E\arabic*)]
\item $u= \pi(\mathrm{face} \, f)$, $v=\pi(\mathrm{edge} \, e)$, for some face~$f$ and edge~$e$ bordering~$f$;
\item $u= \pi(\mathrm{edge} \, e_1)$, $v=\pi(\mathrm{edge} \, e_2)$, for some edges $e_1,e_2$ that share an endpoint.
\end{enumerate}

We define the {\em boundary of~$\calD$} as the set of edges that border the external face of~$\calD$; we denote it by~$\partial\calD$.
The image of~$\partial\calD$ under $\pi$ is called {\em the boundary of  $G$} and denote it by~$\partial G$.

In the next lemma, we show that $G$ is a planar graph (morally, a triangulation) whose distances are at most doubled, compared to the triangular lattice~$\bbT$.
Recall that $\Lambda_r\subset \bbH$ is induced by all vertices bordering faces of the ball of radius $r$ around $\zero$, whereas~$B_r$ is the combinatorial ball around the vertex associated with~$\zero$. A graph is called simple if it has no
 self-loops or multiple edges; it is non-simple otherwise.

\begin{lem}\label{lem:planarity}
	There exists a (non-simple) planar graph~$T$ on $V$ such that:
	\begin{itemize}
		\item each face in~$T$ bordered by some~$u\not\in \partial G$ has degree at most three;
		\item distinct~$u,v\in V$ are adjacent in~$T$ if and only if they are adjacent in~$G$.
	\end{itemize}	
	Additionally, $B_{2r+1}(\pi(f))\supseteq \pi(\Lambda_r)$.
\end{lem}

\begin{proof}
	Define $H$ to be the graph whose vertex set is comprised of all faces of $\calD$ and all edges contained in either $\calD$ or a loop of $\omega$ that intersects $\calD$. Two vertices of $H$ are neighbors in $H$ if they correspond to two edges that share an endpoint, or to a face and an edge that borders it.
	The graph~$H$ inherits from~$\bbH$ a natural embedding in the plane.

	Define~$T$ as the quotient graph of~$H$ with respect to the following relation on its vertices: $u$ and $v$ are identified if they correspond to edges of~$\bbH$ that are on the same loop of $\omega$. 
	We now verify the properties of~$T$.
	The vertex set of~$T$ is $V$. 
	By construction, each equivalence class is a finite connected set, and therefore~$T$ is planar (indeed, one can contract this connected set edge by edge).
	Under the natural embedding, all faces of~$H$ bordered by some~$u\in V\setminus \partial G$ are triangles.
	Since contracting edges can only reduces degrees of faces, this gives the desired bound on the degrees of faces of the quotient graph~$T$. 
	Finally, removing all multiple edges and self-loops from~$T$ gives $G$.

	Consider a face~$g$ of~$\Lambda_r$.
	The distance between~$f$ and~$g$ in~$H$ is at most~$2r$. 
	This implies that the distance in~$H$ between~$f$ and any face or edge in~$\Lambda_r$ is at most~$2r+1$. Since~$G$ is obtained from~$H$ by contraction and deletion of multiple edges and self-loops, the same inequality holds for distances in~$G$.
\end{proof}

\begin{figure}
	\begin{center}
		\includegraphics[width=0.39\textwidth, page=1]{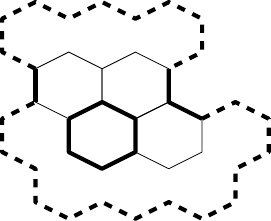}\quad
		\includegraphics[width=0.46\textwidth, page=2]{contraction5.pdf}
		\caption{The construction of the planar graph $G$. {\em Left:} A loop configuration $\omega$. Its intersection with $\calD$ is shown in bold and the edges outside are dotted. {\em Right:} The graph $G$ associated with $\omega$. To construct it, we associate a vertex to each loop of $\omega$ (big black circles), edge of $\calD$ disjoint with $\omega$ (blue dots), and face of $\calD$ (green hexagons). We add an edge between two vertices associated with a face and a bordering edge or loop, or two vertices associated with two edges sharing an endpoint or an edge and a loop sharing an element of $\calD$.}\label{fig:lattices}
	\end{center}
\end{figure}

%\begin{figure}
%	\begin{center}
%		\includegraphics[width=0.39\textwidth, page=1]{contraction4.pdf}\quad
%		\includegraphics[width=0.46\textwidth, page=2]{contraction4.pdf}
%		\caption{{\em Left:} A loop configuration $\omega$. Its intersection with $\calD$ is shown in bold and the edges outside are dotted. {\em Right:} The graph $G$, with vertices associated to loops in black, edges disjoint from $\omega$ in blue, and faces of $\calD$ in green. Vertices with white circles in their center are boundary vertices. The image of the external face of~$\calD$ is shaded.}\label{fig:lattices}
%	\end{center}
%\end{figure}

To each vertex $\ell \in V$ corresponding to a loop in $\omega$, we will associate a subset of its neighbors in $G$, denoted by $I(\ell)$, and having the following properties:
\begin{enumerate}[label=(P\arabic*)]
\item\label{list:f-ell-faces} vertices in $I(\ell)$ correspond to faces of $\calD$,
\item \label{list:f-ell-deg} $|I(\ell)| \leq 5$, and
\item \label{list:f-ell-pairs} for every pair of loops $\ell_1,\ell_2$ such that $\dist_{G}(\ell_1,\ell_2) = 2$, there exists $f \in I(\ell_1) \cup I(\ell_2)$ such that $f$ neighbors both $\ell_1$ and $\ell_2$.
\end{enumerate}

\begin{lem}\label{lem:association}
There exists an association $I(\ell)$ which satisfies  properties \ref{list:f-ell-faces}, \ref{list:f-ell-deg}, \ref{list:f-ell-pairs}.
\end{lem}
\begin{proof}
	We use an auxiliary graph $G_{\mathsf{loop}}$ on the loops of $\omega$ that intersect~$\calD$, where two loops are linked by an edge if they border a common face $f$ of $\calD$. 
	We associate each edge of $G_{\mathsf{loop}}$ with an (arbitrarily chosen) face of $\calD$ that witnesses its existence.
	 The graph $G_{\mathsf{loop}}$ is simple, finite, and planar, and therefore, Euler's formula implies that it must have at least one vertex of degree five or less. Let $\ell$ be the loop corresponding to this vertex, and set $I(\ell)$ to be the faces associated with the edges incident to~$\ell$.
	 We then remove $\ell$ and its neighboring edges, creating a new planar graph; iteration completes the proof.
\end{proof}

For the remainder of the paper, we will fix an association $I$ which satisfies properties~(P1)-(P3). 
Given $n \geq 1$ and $x \in (0,1]$, we define a percolation process $\zeta \in \{0,1\}^{V}$ on the vertices of $G$ as follows:
\begin{itemize}
	\item for any $\ell \in V$ corresponding to a loop of $\omega$, set $\zeta(\ell) = 1$ with probability $(n-1)/n$, independently,
	\item for any $e \in V$ corresponding to an edge of $\calD\setminus \omega$, set $\zeta(e) = 1$ with probability $1- x$, independently, and
	\item for any $f \in V$ corresponding to a face of $\calD$ its state in~$\zeta$ is determined by the state of edges and loops bordering it:
	\begin{itemize}
		\item if $\zeta(e) = 1$ for some~$e$ bordering~$f$, then~$\zeta(f) = 1$;
		\item if $f \in I(\ell)$ and $\zeta(\ell) = 1$ for some loop~$\ell$, then~$\zeta(f) = 1$;
		\item if neither of the above holds, then~$\zeta(f) = 0$.
	\end{itemize}
\end{itemize}

Define $\nu_{G,n,x}$ to be the measure on $\{0,1\}^{V}$ associated with $\zeta$.
The projection of~$\zeta$ on loops of~$\omega$ and edges of~$\calD\setminus\omega$ defines a coherent triplet $(\omega_r,\omega_b,\eta)$: a loop $\ell$ is colored blue if $\zeta(\ell) =1$ and it is colored red otherwise; $\eta(e) = \zeta(e)$ for edges~$e\in \calD\setminus\omega$.

\begin{lem}[From~$\zeta$ to coherent triplets]
\label{lem:zeta-to-triplets}
	The mapping defined above is a bijection between~$\{0,1\}^{V}$ and the set of coherent triplets~$(\omega_r,\omega_b,\eta)$ satisfying~$\omega_r\cup \omega_b = \omega$.
	
	Additionally, the pushforward of~$\nu_{G,n,x}$ under this mapping is~$\mu_{\calD,n,x}^\omega(\cdot \mid \omega_r \cup \omega_b = \omega) $.
\end{lem}

The proof is straightforward and we omit it for brevity.

The process $\zeta$ is a 2-dependent process, in the sense that $\{\zeta(s)\}_{s \in S}$ are independent random variables whenever $S$ is a subset of vertices of $V$ whose pairwise distances are all at least 3. 
We call a vertex $v \in V$ open if $\zeta(v) =1$ and closed otherwise.
For any $A,B\subseteq V$, let $\{A \leftrightarrow B\}\subseteq \{0,1\}^{V} $ denote the event that there is a $\zeta$-open path from some vertex of $A$ to some vertex of $B$.

For a set~$S\subset V$, define its inner vertex boundary as the set of vertices in~$S$ that are adjacent to some vertices in~$V\setminus S$.
The outer boundary of~$S$ is defined as the set of vertices in~$V\setminus S$ that are adjacent to vertices in~$S$.

\begin{prop}\label{prop:PercolationBound}
	For any $n \geq 1$, $x \in (0,1]$, $r \in \bbN$,
	\[
			 \mu_{\calD,n,x}^\omega[\exists \text{ a defect-free circuit in $\calD$ surrounding }  \Lambda_r \mid \omega_r \cup \omega_b = \omega]
			 \geq \nu_{G,n,x}[B_{2r+2}(\pi(\zero)) \not\leftrightarrow  \partial G].
	\]
\end{prop}

\begin{proof}
	The connectivity in~$G$ is the same as in the embedded planar graph~$T$ described in Lemma~\ref{lem:planarity}.
	Therefore, below we discuss this statement for~$T$.
	
	Recall that all faces of~$T$, except those bordered only by vertices in~$\partial G$, have degree at most three.
	The core of the argument is the standard duality for site percolation on a triangulation applied to~$\zeta$ on~$T$.
	Indeed, let~$S$ be the union of connected components in~$\zeta$ that intersect~$B_{2r+2}(\pi(\zero))$.
	If it reaches~$\partial G$, then~$B_{2r+2}(\pi(\zero))$ is connected to~$\partial G$.
	Otherwise, the inner vertex boundary of~$S$ consists of vertices from~$V\setminus \partial G$.
	By the property of~$T$, all faces containing these vertices have degree at most three.
	Hence, the outer boundary of~$S$ forms a circuit of closed vertices.
	This dichotomy implies the following:
\[
\begin{split}
	\nu_{T,n,x}&[B_{2r+2}(\pi(\zero)) \not\leftrightarrow \partial G]  = \nu_{T,n,x}[\exists \text{ a closed circuit surrounding } B_{2r+2}(\pi(\zero))].
\end{split}
\]

	We claim that, if there exists a closed circuit surrounding $B_{2r+2}(\pi(\zero))$, there also exists a closed circuit which surrounds $B_{2r+1}(\pi(f))$ and consists only of vertices that correspond to loops of $\omega$ and edges of $\calD$. To prove this claim, we consider a simple circuit $\calC = \{v_0, \dots, v_N = v_0\}$ of closed vertices. We assume that this circuit is minimal, in the sense that, whenever $v_i, v_j$ are adjacent in $T$, $|i-j| = 1$.
	
	The claim follows if we can show that, whenever $\calC$ includes some vertex $v_i$ that corresponds to a face of $\calD$, we can locally modify $\calC$ in order to bypass $v_i$. 
	Let $N(v_i)$ be the set of neighbors of $v_i$ in $T$.
	Since~$v_i$ corresponds to a face, vertices in~$N(v_i)$ form a cycle of vertices that correspond to loops in~$\omega$ and edges in~$\calD\setminus \omega$ (and not to faces of~$\calD$).
	By minimality, $v_{i-1}$ and $v_{i+1}$ are not neighbors, and $\calC$ does not contain any other element of $N(v_i)$. 
	Therefore, we can partition $N(v_i) \setminus \{v_{i-1},v_{i+1}\}$ into two paths.
	If either of these paths contains only closed vertices, we may replace $v_i$ by this path, possibly decreasing the distance to $\pi(f)$ by one.
	
	Otherwise, there exist two open vertices in~$N(v_i)$ that are separated from one another by $\calC$. By the definition of $\zeta$ on faces of $\calD$, these vertices must correspond to loops of~$\omega$~--- denote them by~$\ell_1,\ell_2$.
	By \ref{list:f-ell-pairs}, there must be a face in $I(\ell_1) \cup I(\ell_2)$ neighboring both $\ell_1$ and $\ell_2$. 
	By definition of $\zeta$, this face is open.  
	Since~$\ell_1$ and~$\ell_2$ are separated from one another by~$\calC$, all faces of~$\calD$ that are adjacent to both of them must belong to~$\calC$
	However, $\calC$ consists of closed vertices.
	This is a contradiction, and the claim follows.

	Finally, assume there exists a closed circuit~$\calC$ which surrounds $B_{2r+1}(\pi(\zero))$ and contains only vertices corresponding to loops of~$\omega$ and edges of~$\calD\setminus\omega$.
	According to the bijection stated in Lemma~\ref{lem:zeta-to-triplets}, the vertices of~$\calC$ correspond to loops of~$\omega_r$ and edges in~$\calD\setminus(\omega\cup \eta)$.
	Viewed as subsets of edges of~$\calD$, their union must contain a defect-free circuit.
	By the second stipulation of Lemma~\ref{lem:planarity}, $\calC$ surrounds $\pi(\Lambda_r)$, whence the defect-free circuit surrounds~$\Lambda_r$.
\end{proof}

The final proposition of this section proves that $\zeta$ can be dominated by a fully independent percolation process. 
For~$\overline{p}:V\to [0,1]$, define~$\mathrm{Ber}_{\overline{p}}$ to be an independent inhomogeneous $\overline{p}$-Bernoulli percolation on~$V$ --- i.e. $\mathrm{Ber}_{\overline{p}}(v) = 1$ with probability $\overline{p}(v)$ and $\mathrm{Ber}_{\overline{p}}(v) = 0$ with probability $1-{\overline{p}}(v)$, for every~$v\in V$, independently. 
If~$\overline{p}\equiv p$, then we denote the process by~$\mathrm{Ber}_p$.
We also use the notation $\preceq_{\mathsf{st}}$ to indicate stochastic domination of measures with respect to the pointwise partial ordering of~$\{0,1\}^{V}$.

\begin{prop}\label{prop:Domination}
	There exists a continuous function $p(n,x): [1,\infty) \times (0,1]$ with $p(1,1) = 0$ such that $\nu_{G,n,x} \preceq_{\mathsf{st}} \mathrm{Ber}_{p(n,x)}$.
\end{prop}

\begin{proof}
	Set~$p_n:=\left(\tfrac{n-1}{n}\right)^{1/6}$ and~$p_x:= (1-x)^{1/3}$.
	For every edge~$e$ of~$\calD\setminus\omega$, define~$S_e\subset V$ to be the set that consists of edge~$e$ and all faces that border it. Further, for every loop~$\ell$ of~$\omega$, define~$S_\ell:=I(\ell)\cup \{\ell\}$.
	We define inhomogeneous independent Bernoulli percolation processes~$\{X_v\}$, indexed by loops of~$\omega$ and edges of~$\calD\setminus\omega$:
	\begin{itemize}
		\item if~$v$ is a loop of~$\omega$, then the percolation parameter equals~$p_n$ on~$S_v$ and~$0$ on its complement;
		\item if~$v$ is an edge of~$\calD\setminus \omega$, then the percolation parameter equals~$p_x$ on~$S_v$ and~$0$ on its complement.
	\end{itemize}
	Next, we define a second set of independent percolation processes~$\{Y_v\}$ with the same index set, which are coupled to $\{X_v\}$ as follows: if~$X_v = 1$ for all $v \in S_v$, set $Y_v = 1$ for all $v \in S_v$; otherwise, set $Y_v = 0$ for all $v \in S_v$. The process $Y_v$ is deterministically set to zero on the complement of $S_v$.  
	
	Set $\bar{X}$ to be the pointwise maximum of $\{X_v\}$ and $\bar{Y}$ to be the pointwise maximum of $\{Y_v\}$. Since $Y_v \leq X _v$, we conclude that $\bar{Y} \preceq_{\mathsf{st}} \bar{X}$. Furthermore, since $|S_v| \leq 6$ when $v$ is a loop of $\omega$ and $|S_v| \leq 3$ when $v$ is an edge of $\calD \setminus \omega$, a union bound implies that $\bar{X} \preceq_{\mathsf{st}} \mathrm{Ber}_{p(n,x)}$, where
	\[
	p(n,x) = \min\{6p_x+3p_n, 1\}.
\]
The proposition follows if we can show that $\zeta$, distributed as $\nu_{G,n,x}$, is stochastically dominated by $\bar{Y}$. Indeed, the probability that $X_v(v) = 1$ is $p_n^{|S_v|} \geq (n-1)/n$ if $v$ is a loop of $\omega$, and $p_x^{|S_v|} \geq 1-x$ if $v$ is an edge of $\calD \setminus \omega$. The desired stochastic domination follows from the construction of $\zeta$. 
\end{proof}

\subsection{Benjamini--Schramm limits}
\label{sec:benjamini-schramm}

%\ag{Edit the section: cut the proof for random graphs}
Instead of analyzing the Bernoulli site percolation process on an abstract finite planar graph, we will study these processes by considering an infinite planar graph that captures the properties of a sequence of graphs $\{G_k\}$ {\em centered around a typical vertex}. The formal framework for doing this is given by Benjamini--Schramm convergence \cite{BenSch01}. In this section, we will recall its definition and collect some useful results about such limits.

Given a sequence of finite graphs $\{G_k\}$, let $\bbU^V_k$ be the uniform measure on the vertices of $G_k$, and let $o_k$ be sampled according to $\bbU^V_k$. We say that $\{G_k\}$ converges in the {\em Benjamini--Schramm} sense~\cite{BenSch01} to the random rooted graph $(G, o)$, whose distribution is denoted by~$\mathbb{P}_\infty$, if, for every fixed $r$ and finite graph $H$,
\[
	\bbU^V_{k} [ B_r(o_k) = H] \to \mathbb{P}_\infty[B_r(o)=H].
\]
We emphasize that $\bbP_\infty$ is a joint distribution on both $G$ and the random root $o$.

Our aim is to state a sufficient condition for the existence of Benjamini--Scramm subsequential limits for a sequence of graphs $\{G_k\}$. This is not trivial, as can be seen from the following example: suppose $G_k$ is a star on $k$ vertices. Then $B_2(v) = G_k$ for every vertex $v \in G_k$, and thus $\bbU^V_{k}[ B_2(o_k) = H]$ converges to zero for every finite $H$. 

In the original paper formulating this notion of convergence~\cite{BenSch01}, it was shown that, whenever $\{G_k\}$ has uniformly bounded degrees, the sequence must have subsequential limits. This condition is too strong for our purposes~--- indeed, degrees of loops are not necessarily bounded. 
Luckily, it can be replaced by a uniform integrability condition on the degrees of $G_k$. More precisely, we say that $\{G_k\}$ has uniformly integrable degrees if  
\[
	\lim_{r \to \infty} \sup_{k \to \infty} \tfrac{1}{|V_k|}\sum_{v \in V_k} \deg(v)\1_{\deg(v)\geq r}  = 0,
\]
where~$\deg_{G_k}(v)$ is the degree of $v$~$G_k$.

\begin{prop}\label{prop:BSconvergence}\cite[Theorem 3.1]{BenLyoSch15}
If $\{G_{k}\}$ has uniformly integrable degrees, then there exists a subsequence that converges in the Benjamini-Schramm sense.
\end{prop}

\begin{rem}
	We observe that uniform integrability of degrees is not a necessary condition for the existence of subsequential limits. To see this, consider the sequence $\{G_k\}$ given by setting $G_k$ to be the disjoint union of a cycle on $k - \sqrt{k}$ vertices with a clique on $\sqrt{k}$ vertices. This sequence does not have uniformly integral degrees, but converges to $\mathbb{Z}$, rooted at the origin, in the sense of Benjamini--Schramm.
\end{rem}

In light of Proposition~\ref{prop:PercolationBound}, we will also need to consider the boundary of the planar graphs $G$ described in the previous section. For~$k\geq 1$, let $G_k$ be a graph with vertex-set~$V_k$ and let~$\partial G_k\subset V_k$ be a nonempty subset of vertices that we call the boundary.
The sequence~$\{G_k\}$ is called {\em F{\o}lner} with respect to $\{\partial G_k\}$ if
\[
	\lim_{k \to \infty} \frac{ |\partial G_k|}{|V_k|} = 0.
\]
Note that, if a sequence of graphs is F{\o}lner, $|V_k| \to \infty$. A corollary of Proposition~\ref{prop:BSconvergence} involves the typical distance from the root to the boundary of a F{\o}lner sequence of graphs.
\begin{cor}\label{cor:boundary}
	Let $\{G_k\}$ be a F{\o}lner sequence with respect to $\{\partial G_k\}$ with uniformly integrable degrees. Then, for any integer $r \geq 1$,
	\[
		\lim_{k \to \infty} \bbU^V_k [ \dist_{G_k}(o_k, \partial G_k) \leq r] = 0.
	\]
\end{cor}

\begin{proof}
	Assume, for the sake of contradiction, that there exist $\varepsilon >0$ such that, for infinitely many $k$'s,
	\[
		\bbU^V_k[ \dist_{G_k}(o_k, \partial G_k) \leq r] > \varepsilon.
	\]
	This implies that, for infinitely many $k$'s,
	\[
		\left|\bigcup_{u \in \partial G_k} B_r(u) \right| > \varepsilon |V(G_k)|.
	\]
	Now, consider the auxiliary graph sequence $\{\tilde{G}_k\}$, given by adding one vertex to $G_k$ and connecting it to every vertex in $\partial G_k$. No subsequene of this sequence can converge in the Benjamini--Schramm sense, as, with probability at least $\varepsilon$, a ball of radius $2r + 2$ around $o_k$ has volume at least $\varepsilon |V(G_k)|$. Since $\{G_k\}$ is F{\o}lner with respect to $\{\partial G_k\}$,  $|V(G_k)| \to \infty$, meaning $B_{2r+2}(o_k)$ cannot converge to any finite graph with positive probability.  However, since $\{G_k\}$ has uniformly integrable degrees and is F{\o}lner with respect to $\{\partial G_k\}$, the sequence $\{\tilde{G}_k\}$ must also have uniformly integrable degrees. This contradicts Lemma~\ref{prop:BSconvergence}, as required.
\end{proof}

The final ingredient is a statement about Bernoulli site on Benjamini--Schramm limits of finite planar graphs.

\begin{thm}\label{thm:Ron}\cite{Pel19}
There exists a $p_0 >0$ such that the following holds. Let $\{G_k\}$ be a sequence of finite simple planar graphs that converges to a limit $(G,o)$ in the sense of Benjamini--Schramm.  Then
	\[
		\mathrm{Ber}_{p_0}\left( \exists \text{ an infinite }  \text{open path in } G \right) = 0,
	\]
	where~$\mathrm{Ber}_p$ is the independent $p$-Bernoulli percolation on vertices of~$G$.
\end{thm}
We emphasize that the choice of $p_0$ does not depend on the sequence of graphs $\{G_k\}$.

\subsection{Proof of Proposition \ref{prop:PlanarGraphProp}}
\label{sec:xor-circuit-proof}

Let $\bbP_{n,x}$ be a translation-invariant Gibbs measure for the loop $O(n)$ model, where the parameters satisfy \eqref{eq:theorem-1-regime}, and $\mu_{n,x}$ be its associated defect representation.

Recall that~$\Lambda_k$ denotes the ball of radius~$k$ around~$\zero$ defined in Section~\ref{sec:intro}.
For~$k\geq 1$ and a loop configuration~$\omega$, define~$G_k(\omega)$ as the graph~$G$ introduced in Section~\ref{sec:graph-construction} for~$\Lambda_k$ and~$\omega$.
In the proofs below, we do not explicitly mention the dependence of this graph on~$\omega$ and simply write~$G_k$.

\begin{lem}
\label{lem:uniform-integrability}
	Assume that $\bbP_{n,x}$ exhibits no bi-infinite paths almost surely.
	Then, the sequence of graphs~$\{G_k\}$ has uniformly integrable degrees, for~$\bbP_{n,x}$-almost every~$\omega$.
\end{lem}

\begin{proof}
	Denote the set of vertices of~$G_k$ by~$V_k$ and the set of faces of~$\Lambda_k$ by~$F_k$.
	
	The degree of a loop~$\ell$ of~$\omega$ in~$G_k$ is bounded above by the number of faces of~$\Lambda_k$ bordered by~$\ell$ plus the number of edges in~$\Lambda_k\setminus\omega$ that have an endpoint on~$\ell$.
	Thus, the degree of~$\ell$ is smaller or equal than three times the number of faces of~$\Lambda_k$ bordered by~$\ell$.
	Summing the degrees of loops in~$\omega$ that intersect~$\Lambda_k$, we will count every face of~$\Lambda_k$ at most three times.
	Since each edge of~$\Lambda_k\setminus\omega$ or face of~$\Lambda_k$ has degree at most~$6$ in~$G_k$, we find that, for any~$r>6$, 
	\begin{equation}\label{eq:length-bound}
\sum_v \deg(v)\, \1_{\deg(v)\geq r}
			\leq  9 \cdot |\{\text{faces of } \Lambda_k \text{ bordering a loop with degree} \geq r\}|.
	\end{equation}
	By the above, the degree of~$\ell$ is smaller or equal than three times its length.
	Define~$A_r(f,\omega)$ as the event that a face~$f$ borders a loop in $\omega$ of length at least~$r/3$.
	Since~$|V_k|\geq |F_k|$, we obtain
	 \[
		\tfrac{1}{|V_k|}\sum_v \deg(v)\1_{\deg(v)\geq r}\leq 9 \cdot \bbU_k^F \left(A_r({\bf f}_k,\omega) \right),
	\]
	where $\bbU_k^F$ is the uniform measure on the faces of $\Lambda_k$, independent of~$\omega$, and~${\bf f}_k \sim \bbU_k^F$.
	
%	\ag{Prove the next equality. It's true for $\lim \limsup \bbE (\bbU[...])$. Can we show it for  $\bbE[\lim \limsup  (\bbU[...])$? If so, it's enough to consider only deterministic graphs (fix~$\omega$). Otherwise, we have to average over~$\omega$, so we'll require Section~4.3}
	
	It is standard (and easy to prove) that the supremum in the definition~\eqref{eq:ui-def} of the uniform integrability can be replaced by the limit superior.
	Thus, it remains to show that for~$\bbP_{n,x}$-almost every~$\omega$, we have
	\[
		\lim_{r \to \infty}  \limsup_{k \to \infty} \, \bbU_k^F \left(A_r({\bf f}_k,\omega) \right) = 0.
	\]
%	We first note that it is enough to show that
%	\begin{equation}\label{eq:ui-limsup}
%		\lim_{r \to \infty}  \limsup_{k \to \infty} \, \bbU_k^F \left(A_r({\bf f}_k,\omega) \right) = 0.
%	\end{equation}
%	Indeed, then, for each~$\varepsilon>0$, there exists~$r_0$ such that, for any~$r>r_0$, we have
%	\[
%		\limsup_{k \to \infty} \, \bbU_k^F \left(A_r({\bf f}_k,\omega) < \varepsilon
%	\]
%	Thus, there exists~$k_0$ such that, for any~$k>k_0$, we have
%	\[
%		\bbU_k^F \left(A_r({\bf f}_k,\omega) < 2\varepsilon
%	\]
%	We may then choose~$r_1$ sufficiently large so that if $r>r_1$, then the last display holds also for~$k\leq k_0$. 
%	In conclusion, for~$r>\max\{r_0, r_1\}$, we have
%	\[
%		\sup_k \bbU_k^F \left(A_r({\bf f}_k,\omega) < 2\varepsilon,
%	\]
%	which implies~\eqref{eq:ui-sup}.
%	We now focus on proving~\eqref{eq:ui-limsup}.
	For $f\in F_k$, define~$\tau_f$ as the translation of the triangular lattice such that~$\tau_f(\zero)=f$. Then,
	\[
		\bbU_k^F \left(A_r({\bf f}_k,\omega) \right) = \tfrac{1}{|F_k|}\sum_{f\in F_k}\1_{A_r(\zero,\tau_f^{-1}\omega)}.
	\]
	By translation invariance of~$\bbP_{n,x}$ and the ergodic theorem, we obtain that there exists a random variable~$Z_r(\omega)$ such that, for~$\bbP_{n,x}$-almost every~$\omega$,
	\[
		\lim_{k \to \infty}\bbU_k^F \left(A_r({\bf f}_k,\omega)\right) = Z_r(\omega).
	\]
	Furthermore, $Z_r(\omega)$ is invariant to translations.
	Indeed, define~$\tau_1$ as the translation of the triangular lattice by one to the right.
	Then~$Z_r(\omega) = Z_r(\tau_1\omega)$, since
	\[
		|\tfrac{1}{|F_k|}\sum_{f\in F_k}\1_{A_r(\zero,\tau_f^{-1}\omega)}
		-\tfrac{1}{|F_k|}\sum_{f\in F_k}\1_{A_r(\zero,\tau_f^{-1}\tau_1\omega)}|
		\leq \frac{|F_k \triangle \tau_1 F_k|}{|F_k|} \xrightarrow[k\to\infty]{} 0.
	\]
	Appealing to the ergodic theorem again, we get an expression for the expectation over~$\bbP_{n,x}$:
	\[
		\bbE_{n,x}(Z_r(\omega)) = \bbP_{n,x}(A_r(\zero,\omega))
	\]
	Take any~$\varepsilon>0$. Since~$Z_r$ is decreasing in~$r$,
	\begin{align*}
		\bbP_{n,x}(\lim_{r\to\infty} Z_r(\omega) > \varepsilon) 
		&= \lim_{r\to\infty} \bbP_{n,x}(Z_r(\omega) > \varepsilon)\\
		&\leq \tfrac{1}{\varepsilon}\lim_{r\to\infty} \bbE_{n,x}(Z_r(\omega)) 
		=  \tfrac{1}{\varepsilon} \lim_{r\to \infty} \bbP_{n,x}(A_r(\zero,\omega)).
	\end{align*}
	The limit on the right-hand side equals the probability that~$\zero$ borders a bi-infinite path, which is zero by our assumption.
\end{proof}

We are now ready to prove Proposition~\ref{prop:PlanarGraphProp}, which implies Theorem~\ref{thm:long-around-1-inf-vol}, as we have already shown.

\begin{proof}[Proof of Proposition~\ref{prop:PlanarGraphProp}]
	For integers $r,k$ and a face $f$, let $\mathrm{Circ}_{r,k}(f)$ be the event that there exists a defect-free circuit in $\Lambda_k$ which surrounds $\Lambda_r(f)$. 
	As in the proof of Lemma~\ref{lem:uniform-integrability}, we set~$\bbU_k^F$ to be the uniform measure on faces of~$\Lambda_k$ and sample~${\bf f}_k$ from~$\bbU_k^F$, independently of $\omega$.
	By translation invariance, inclusion of events, and the Gibbs property, we have that, for any integer $k$,
	\begin{align*}
		\mu_{n,x} [\exists \text{ a defect-free circuit} \text{ surrounding } &\Lambda_r]
		 \geq  (\mu_{n,x}\otimes \bbU_k^F)\left[\mathrm{Circ}_{r,k}(\mathbf{f}_k)\right] \\ & = \bbE_{n,x} \left[(\mu_{\Lambda_k, n,x}^{\omega}\otimes \bbU_k^F)\left[\mathrm{Circ}_{r,k}(\mathbf{f}_k) \mid \omega_r \cup \omega_b = \omega\right]\right].
	\end{align*}
By Fatou's lemma, the proposition will be proved if we can show that, for $\bbP_{n,x}$-almost every~$\omega$, 
\[
\lim_{k} \,  (\mu_{\Lambda_k, n,x}^{\omega}\otimes \bbU_k^F)\left[\mathrm{Circ}_{r,k}(\mathbf{f}_k) \mid \omega_r \cup \omega_b = \omega\right] = 1.
\]
To that end, for the remains of the proof, we fix~$\omega$ sampled from~$\bbP_{n,x}$.
By Lemma~\ref{lem:uniform-integrability}, we can assume that the degrees of~$\{G_k\}$ are uniformly integrable.

	Define~$\bbU_k^{V}$ to be the uniform measure on vertices of~$G_k$ and consider~$o_k$ to be distributed as~$\bbU_k^{V}$.
	By Lemma~\ref{lem:uniform-integrability} and Proposition~\ref{prop:BSconvergence}, there exists an infinite planar rooted graph $(G,o)$ that is a subsequential Benjamini--Schramm limit of this sequence.
	Let~$\bbE_{\infty}$ denote the expectation operator associated with the joint distribution of~$(G,o)$.
	
	Fix~$p_0\in (0,1)$ that will be determined later.
	Define~$\mathrm{Ber}_{p_0}^k$, $k\geq 1$ (resp. $\mathrm{Ber}_{p_0}$) as a Bernoulli site percolation on~$G_k$ (resp. $G$) of parameter~$p_0$.
	
	Take~$r\geq 1$ and~$R\geq 2r+2$.
	Since the event $B_{2r+2}\left(o_k\right) \leftrightarrow B_{R}^c(o_k)$ is local, in the sense that it is measurable with respect to the restriction of the Bernoulli percolation to a ball of radius $R+1$ around $o_k$, we find that
	\[
		\lim_{k \to \infty} \, 
		(\bbU_k^V\otimes\mathrm{Ber}_{p_0}^k) \left[B_{2r+2}\left(o_k\right) \leftrightarrow B_{R}^c\left(o_k\right) \right] = 		
		\bbE_\infty\left[ \mathrm{Ber}_{p_0} \left[B_{2r+2}\left(o\right) \leftrightarrow B_{R}^c\left(o\right) \right]\right],
	\]
	where the limit may be taken along a converging subsequence.
	
	Since the event on the right-hand side is decreasing in $R$, the limit of the probability is equal to the probability of the infinite intersection --- namely, the probability that $B_{2r+2}(o)$ intersects an infinite open path. 	
	By Theorem~\ref{thm:Ron} from~\cite{Pel19}, we may choose~$p_0>0$ sufficiently small, so that this probability is zero, whence
	\[
		\lim_{R\to \infty}\lim_{k \to \infty} \, (\bbU_k^V\otimes\mathrm{Ber}_{p_0}^k) \left[B_{2r+2}\left(o_k\right) \leftrightarrow B_{R}^c\left(o_k\right) \right] = 0.
	\]
	Thus, for any~$\varepsilon,r>0$, we may choose $R$ large enough so that, for all $k$ large enough,
	\begin{equation}\label{eq:bound-connection}
		(\bbU_k^V\otimes\mathrm{Ber}_{p_0}^k) \left[B_{2r+2}\left(o_k\right) \leftrightarrow B_{R}^c\left(o_k\right) \right] < \varepsilon.
	\end{equation}
	Next, we recall from Section~\ref{sec:graph-construction} that $\partial G_k$ is the set of vertices of $G_k$ that are in the image of boundary edges of $\Lambda_k$. It is straightforward to check that the sequence $\{G_k\}$ is  F{\o}lner with respect to $\{\partial G_k\}$.  Thus, Corollary~\ref{cor:boundary} allows us to deduce that, for the $\varepsilon$ and $R$ chosen above, we can possibly increase $k$ to ensure that
	\begin{equation}\label{eq:bound-boundary}
		\bbU_k^V \left[\dist_{G_k}(o_k, \partial G_k) \leq R\right] < \varepsilon.
	\end{equation}

We now return to the task of bounding $(\mu_{\Lambda_k, n,x}^{\omega}\otimes \bbU_k^F)\left[\mathrm{Circ}_{r,k}(\mathbf{f}_k) \mid \omega_r \cup \omega_b = \omega\right]$ from below. 

	Recall~$p(n,x)$ from Proposition~\ref{prop:Domination}.
	We pick $\delta$ sufficiently small so that $p(n,x) \leq p_0$ for any $(n,x)\in [1,1+\delta]\times [1-\delta,1]$.
	Combining Propositions \ref{prop:PercolationBound} and \ref{prop:Domination}, we get,
	\begin{align*}
		\mu^{\omega}_{\Lambda_k(\zero),n,x}&[\mathrm{Circ}_{r,k}({\bf f}_k) \mid \omega_r\cup\omega_b=\omega ] \\
		&\geq 
		{\sf Ber}_{p(n,x)}[B_{2r+2}(\pi({\bf f}_k)) \not\leftrightarrow \partial G_k] \\
		&\geq 
		1 - {\sf Ber}_{p_0}[B_{2r+2}(\pi_\omega({\bf f}_k)) \leftrightarrow B_R^c(\pi({\bf f}_k))] - \1_{\mathrm{dist}_{G_k} (\pi({\bf f}_k),\partial G_k) < R}.
	\end{align*}
	We now aim to take the expectation of the final line over $\bbU_k^V$ and bound it by~\eqref{eq:bound-connection} and~\eqref{eq:bound-boundary}.
	One technical point here is that the equations deal with a vertex~$o_k$ chosen uniformly among {\em all} of~$V_k$, while in the last equation, $\pi({\bf f}_k)$ is chosen uniformly among elements of~$V_k $ that correspond to faces in~$\Lambda_k$.
	This issue can be easily resolved by noting that the latter set constitutes at least a constant portion of~$V_k$ and thus $\pi(\mathbf{f}_k)$ is uniformly absolutely continuous with respect to $o_k$.
	
	In conclusion, the last chain of inequalities holds true for~$o_k$ in the place of~$\pi(\mathbf{f}_k)$ at the cost of a multiplicative constant~$c$:
	\[
		\mu^{\omega}_{\Lambda_k,n,x}[\mathrm{Circ}_{r,k}({\bf f}_k) \mid \omega_r\cup\omega_b=\omega ] 
		\geq 
		1 - c\, ({\sf Ber}_{p_0}[B_{2r+2}(o_k) \leftrightarrow B_R^c(o_k)] +\1_{\mathrm{dist}(o_k,\partial G_{k}) < R}).
	\]
	Taking expectation over~${\bf f}_k$, the bounds ~\eqref{eq:bound-connection} and~\eqref{eq:bound-boundary} imply that, when $k$ is sufficiently large, 
	\[
		\mu^{\omega}_{\Lambda_k(\zero),n,x} [\mathrm{Circ}_{r,k}({\bf f}_k) \mid \omega_r\cup\omega_b=\omega] \geq 1 - 2c\varepsilon.
	\]
	Since~$\varepsilon>0$ is arbitrary, this finishes the proof.
\end{proof}

\subsection{Proofs of Corollaries \ref{cor:RSW} and \ref{cor:long-around-1}}

Before we prove the remaining corollaries, we must address a technical point: we have not shown that there exists any translation-invariant Gibbs measure for the loop~$O(n)$ model.
Usually, one constructs Gibbs measures by taking a (subsequential) thermodynamic limit of finite volume measures. 
However, since the loop~$O(n)$ model can have arbitrarily long range dependencies, it is theoretically possible for limiting measures not to satisfy the Gibbs property. 
For example, \cite{haggstrom1996random} shows that this occurs for the random-cluster measure on $d$-regular trees whenever $d \geq 3$ and $q>2$. 

In the next proposition, we exclude such pathology.
The proof relies on the uniqueness results of Burton and Keane \cite{BurKea89}.

	We let $\{\xi_k\}$ be a sequence of loop configurations, and recall that $\mathbf{f}_k$ is a uniformly chosen face in $\Lambda_k(\zero)$. Let $\tau_f$ be the unique translation map of $\bbT$ which maps $\zero$ to~$f$. 
	Define
	\[
		\tilde{\bbP}^{\, \xi_k}_{k,n,x}(\cdot) := \bbP^{\, \xi_k}_{\Lambda_k(\zero),n,x} \left( \tau^{-1}_{\mathbf{f}_k} \cdot\right),
	\]
	where $\mathbf{f}_k$ is chosen independently of the loop configuration. In words, $\tilde{\bbP}^{\, \xi_k}_{k,n,x}$ samples a loop configuration from $\bbP_{\Lambda_k(\zero),n,x}^{\, \xi_k}$, and then recenters the configuration around a uniformly chosen face.
	
\begin{prop}\label{prop:gibbs}
	For any $n,x>0$ and any sequence of loop configuration~$\{\xi_k\}$, any subsequential limit of $\{\tilde{\bbP}^{\, \xi_k}_{k,n,x}\}$ is a translation-invariant Gibbs measure with at most one bi-infinite path.
\end{prop}

	We note that, unlike all other statements in this paper, this proposition holds regardless of the value of $n$ and $x$. A similar statement appears in \cite[Theorem 4.31]{Gri06}, in the context of the random cluster model, and in \cite[Section 6.3]{LamTas19}, as a corollary of a more general framework. We include a proof for completeness.
	
\begin{proof}
	We fix~$n,x>0$ and omit them from the notation for brevity.
	Supposed that $\{\tilde{\bbP}^{\xi_k}_{k}\}$ converges to an infinite-volume measure $\bbP$. 
	We claim that $\bbP$ is translation-invariant. It is enough to show this invariance for the translation by one to the right, which we denote by $\tau_1$. Let $A$ be an arbitrary event, and consider the difference  $\tilde{\bbP}^{\xi_k}_{k}(A) - \tilde{\bbP}^{\xi_k}_{k}(\tau_1 A)$. By definition,
\begin{align*}
		\left| \tilde{\bbP}^{\xi_k}_{k}(A) -\tilde{\bbP}^{\xi_k}_{k}(\tau_1 A) \right|	
		&=\left|\bbP^{\xi_k}_{\Lambda_k(\zero)}(\tau_{\mathbf{f}_k}A) - \bbP^{\xi_k}_{\Lambda_k(\zero)}(\tau_{\mathbf{f}_k} \tau_1 A) \right| \\
		&\leq \tfrac{1}{|\{\text{faces in } \Lambda_k(\zero)\}|}\cdot \mathbb{E}^{\xi_k}_{\Lambda_k(\zero)} \left[ \left| \sum_{f\in F(\calD_k)} \1_{\tau_f A}-\sum_{f'\in F(\tau_1\Lambda_k(\zero))}\1_{\tau_{f'} A}\right|\right]\\
		&\leq \tfrac{|\{\text{inner or outer boundary faces of }\Lambda_k(\zero)\}|}{|\{\text{faces in } \Lambda_k(\zero)\}|} \xrightarrow[k\to \infty]{}  0,
	\end{align*}
	where the last inequality follows from the fact that all faces except those on the boundary are counted in both sums.

Now, let $\omega$ be distributed as $\bbP$, and define $\sigma$ to be the following configuration on $\{-1,1\}^{\bbT}$: $\sigma_\zero$ is $\pm 1$ with probability~$1/2$ and, for any neighboring faces~$f,g$, we have~$\sigma_f=\sigma_g$ if and only if their common edge is not in~$\omega$. 
In Section \ref{sec:ising}, we used this construction to associate the loop $O(1)$ model with the Ising model; for more general values of $n$, the distribution of $\sigma$ is non-local.

By definition, the pushforward measure on $\sigma$ is translation-invariant. Furthermore, the measure has the uniform-finite energy property, in the sense that the probability $\sigma_f = +1$ is uniformly bounded away from zero and one even if one conditions on the state of all other spins. Indeed, switching the sign of $\sigma_f$ is equivalent to mapping $\omega$ to $\omega \XOR \Gamma_f$, where $\Gamma_f$ is the loop of length six surrounding $f$. This operation changes the number of loops by at most two, and the number of edges by at most six, whence
\[
	\bbP^{\xi_k}_{\Lambda_k(\zero)}(\omega \XOR \Gamma_f)\geq  c(n,x) \cdot \bbP^{\xi_k}_{\Lambda_k(\zero)}(\omega), \text{ where }
	c(n,x) = \min \{n^2,n^{-2}\} \cdot \min \{x^6, x^{-6}\}.
\]

Using  the standard Burton-Keane argument \cite[Theorem 2]{BurKea89}, this implies that there is at most one infinite cluster of $+1$'s in $\sigma$, $\bbP$-almost surely; by symmetry, the same holds for $-1$'s. The correspondence between $\sigma$ and $\omega$ implies that a bi-infinite path in $\omega$ will occur as an interface between an infinite cluster of $+1$'s and an infinite cluster of $-1$'s. Since there is at most one of each, there can only be one such interface.

To complete the proof, we are left to prove that $\bbP$ is a Gibbs measure for the loop $O(n)$ model with edge weight $x$ --- that is, to confirm that the marginal of $\bbP$ on any finite subgraph $\calD$, conditioned on the restriction of $\omega$ to $\calD^c$, is $\bbP_{\calD}^{\omega}$. Fix $\calD$ to be a finite subgraph containing $\zero$, and let $R$ be an integer such that $\Lambda_R(\zero)$ contains $\calD$. We set $\calT_R$ to be the event that $\omega$ contains at most two disjoint paths that intersect both $\calD$ and $\Lambda_R(\zero)^c$. Since $\bbP$ is supported on configurations with at most one bi-infinite path, we have that
\begin{equation}\label{eq:fRprob}
\lim_{R \to \infty} \bbP[\calT_R] = 1.
\end{equation}
Fix some loop configuration $\omega_0$. 
Then, using Levy's upwards theorem and \eqref{eq:fRprob}, and the definition of $\bbP$, we find
\begin{align*}
\bbP( \omega = \omega_0 \text{ on } \calD \, | \,  \omega \text{ on } \calD^c ) & = \lim_{R \rightarrow \infty} \bbP( \omega = \omega_0 \text{ on } \calD  \, | \,  \omega \text{ on } \Lambda_R(\zero) \setminus \calD) \\ & =\lim_{R \rightarrow \infty} \bbP( \omega = \omega_0 \text{ on } \calD  \, | \,  \omega \text{ on } \Lambda_R(\zero) \setminus \calD, \, \calT_R) \\ & = \lim_{R \rightarrow \infty} \lim_{k \to \infty} \,  \tilde{\bbP}^{\xi_k}_{k}( \omega = \omega_0 \text{ on } \calD  \, | \,  \omega \text{ on } \Lambda_R(\zero) \setminus \calD, \calT_R).
\end{align*}

Let $\calT^{(i)}_R$ be the event that~$\omega$ contains precisely $i$ disjoint crossings from $\calD$ to $\Lambda_R(\zero)^c$. Since the loop~$O(n)$ model is supported on loop configurations, $\calT^{(i)}_R = \emptyset$ when~$i$ is odd. Thus, $\calT_R = \calT^{(0)}_R \cup \calT^{(2)}_R$. If $\calT^{(0)}_R$ occurs, then every loop that intersects $\calD$ is contained in $\Lambda_R(\zero)$; in particular, the restriction of $\tilde{\bbP}^{\xi_k}_{k}$ to $\calD$ is determined by the restriction of $\omega$ to $\Lambda_R(\zero)$.

Recall that~$\tilde{\bbP}^{\xi_k}_{k}$ is defined by sampling a loop configuration, distributed according to~$\bbP_{\Lambda_{k}(\zero)}^{\xi_k}$, and recentering it around a uniformly chosen face~$\mathbf{f_k}$.
By definition of the loop~$O(n)$ measure on finite domains, we deduce that, for any $k$,
\begin{equation}\label{eq:gibbs-t-r-0}
	\tilde{\bbP}^{\xi_k}_{k}( \omega = \omega_0 \text{ on } \calD  \, | \,  \omega \text{ on } \Lambda_R(\zero) \setminus \calD, \calT^{(0)}_R, \dist (\mathbf{f_k}, \Lambda_k(\zero)^c)>R) = \frac{x^{|\omega_0|} n^{\ell(\omega_0)}}{Z^{\omega}_{\calD}} = \bbP^{\omega}_{\calD}(\omega_0),
\end{equation}
where $|\omega_0|$ is the number of edges in $\omega_0 \cap \calD$, and $\ell(\omega_0)$ is the number of loops intersecting $\calD$ (as there are no bi-infinite paths intersecting $\calD$ under our assumptions). Note that the condition on~$\mathbf{f_k}$ ensures that none of the edges in~$\calD$ belong to~$\Lambda_k(\mathbf{f_k})^c$ --- that is, none of these edges are fixed by the boundary conditions~$\xi_k$.

The relation~\eqref{eq:gibbs-t-r-0} is precisely the desired Gibbs condition, modulo the restriction on~$\mathbf{f_k}$. Luckily, for any fixed $R$, the probability that $\dist (\mathbf{f_k}, \Lambda_k(\zero)^c)>R$ approaches 1 as $k$ tends to infinity, whence the statement follows.

Next, we assume that $\calT^{(2)}_R$ occurs. In this case, there is either one bi-infinite path which intersects $\calD$, or precisely one loop which intersects both $\calD$ and $\Lambda_R(\zero)^c$. Luckily, the marginal of the loop~$O(n)$ model on $\calD$ is the same in both cases: for any $k > R$,
\begin{equation}\label{eq:InfinitePath}
	\tilde{\bbP}^{\xi_k}_{k}( \omega = \omega_0 \text{ on } \calD  \, | \,  \omega \text{ on } \Lambda_R(\zero) \setminus \calD, \calT^{(2)}_R,  \dist (\mathbf{f_k}, \Lambda_k(\zero)^c)>R) = \frac{x^{|\omega_0|} n^{\ell_R(\omega_0) + 1}}{Z^{\omega}_{\calD}},
\end{equation}
where $\ell_R(\omega_0)$ is the number of loops that intersect $\calD$ and are contained in $\Lambda_R(\zero)$. Since the righthand side of \eqref{eq:InfinitePath} is independent of $k$, we may take the limit in $k$ with impunity. As above, the restriction on~$\mathbf{f_k}$ disappears; the limit in $R$ of $\calT^{(2)}_R$ is precisely the event that there exists a unique bi-infinite path intersecting $\calD$; similarly, the limit in $R$ of $\ell_R(\omega_0)$ is the number of finite loops that intersect $\calD$. Thus,
\[
\lim_{R \to \infty} \frac{x^{|\omega_0|} n^{\ell_R(\omega_0) + 1}}{Z^{\omega}_{\calD}} =  \frac{x^{|\omega_0|} n^{\ell(\omega_0)}}{Z^{\omega}_{\calD}} = \bbP^{\omega}_{\calD}(\omega_0).
\]
\end{proof}

\begin{proof}[Proof of Corollary \ref{cor:RSW}]
	By Theorem \ref{thm:long-around-1-inf-vol} and  Proposition \ref{prop:gibbs}, there exists at least one translation-invariant Gibbs measure with either (I) infinitely many loops surrounding every face or (II) a bi-infinite path. By the dichotomy theorem~\cite[Theorem 2]{DumGlaPel21}, when~$n\geq 1$ and~$x\leq \tfrac1{\sqrt{n}}$, (I) implies the RSW estimate~\eqref{eq:RSW} and (II) is impossible.
\end{proof}

\begin{proof}[Proof of Corollary \ref{cor:long-around-1}]
By definition, the probability of the events $C_{r,R}(\mathbf{f}_k)$ and $S_{r,R}(\mathbf{f}_k)$ under the product measure of $\bbP^{\xi_k}_{\Lambda_k(\zero),n,x}$ and $\mathbf{f}_k$ is the same as the probability of $C_{r,R}(\zero)$ and $S_{r,R}(\zero)$ under $\tilde{\bbP}^{\xi_k}_{k,n,x}$. Assume that $\tilde{\bbP}^{\xi_k}_{k,n,x}$ converges to an infinite-volume $\bbP_{n,x}$.
Since the events $C_{r,R}(\zero)$ and $S_{r,R}(\zero)$ are mutually exclusive,
\begin{align}
	\lim_{r \to \infty} \lim_{R \to \infty} \bbP_{n,x} ( C_{r,R}(\zero) \cup S_{r,R}(\zero)) & = \lim_{r \to \infty} \lim_{R \to \infty} \bbP_{n,x} ( C_{r,R}(\zero)) +  \bbP ( S_{r,R}(\zero)) \nonumber \\
	& = \bbP_{n,x} \left(\bigcup_{r >0} \, \bigcap_{R > 0 } C_{r,R}(\zero) \right) + \bbP_{n,x} \left(\bigcap_{r >0} \, \bigcup_{R > 0 } S_{r,R}(\zero)\right), \label{eq:c-or-s-limit}
\end{align}
where we may take the limit inside the measure since $C_{r,R}(\zero)$ is decreasing in $R$ and increasing in $r$, whereas $S_{r,R}(\zero)$ is increasing in $R$ and decreasing in $r$.
One has
\begin{align*}
\bigcup_{r >0} \, \bigcap_{R > 0 } C_{r,R}(\zero)& = \{\exists \text{ a bi-infinite path} \}, \text{ and } \\ \bigcap_{r >0} \, \bigcup_{R > 0 } S_{r,R}(\zero) & = \{\text{every face is surrounded by infinitely many loops} \}.
\end{align*}
Thus, by Proposition \ref{prop:gibbs} and Theorem \ref{thm:long-around-1-inf-vol}, the RHS of~\eqref{eq:c-or-s-limit} equals 1, as required.

For the second stipulation, we wish to show that $\{\calL(\mathbf{f}_k)\}$ is not uniformly integrable under  $\tilde{\bbP}^{\xi_k}_{\Lambda_k(\zero),n,x}$ (which equals $\bbP$). Assume that, for some $r >0$,
\begin{equation}\label{eq:assumption-c-r-R-non-unif-int}
	\lim_{R \to \infty} \liminf_{k \to \infty} \,  \tilde{\bbP}^{\xi_k}_{\Lambda_k(\zero),n,x}(C_{r,R}(\mathbf{f}_k)) > 0.
\end{equation}
Fix a loop configuration $\omega$, and let $\mathbf{L}_s$ be the set of faces in $\Lambda_k(\zero)$ which border a loop of length greater than $s$. We also set $\mathbf{C}_{r,R}$ to be the set of faces for which $C_{r,R}(f)$ occurs. Every $ f \in \mathbf{C}_{r,R}$ must be at distance at most $r$ from a face bordering a loop intersecting the complement of $\Lambda_R(f)$. Thus, it is straightforward to see that
\[
\mathbf{C}_{r,R} \subset \bigcup_{f \in \mathbf{L}_{(R-r)/6} } \{\text{faces of } \Lambda_r(f) \}.
\]
Taking a union bound and then averaging over $\omega$, we find that, for any $R > r$,
\[
\tilde{\bbP}^{\xi_k}_{\Lambda_k(\zero),n,x} [C_{r,R}(\mathbf{f}_k)] \leq |\{\text{faces of } \Lambda_r(\zero)\} | \cdot  \mathbb{P}[\mathcal{L}(\mathbf{f}_k )> (R-r)/6].
\]
Then, our assumption~\eqref{eq:assumption-c-r-R-non-unif-int} implies that,
\[
\lim_{R \to \infty} \liminf_{k \to \infty} \tilde{\bbP}^{\xi_k}_{\Lambda_k(\zero),n,x}(\mathcal{L}(\mathbf{f}_k) > R) > 0,
\]
that is $\{\calL(\mathbf{f}_k)\}$ has a uniformly positive probability to be larger than any $R$. Thus, the sequence is not tight, and therefore not uniformly integrable.

By the first part of the corollary, we may now assume that, for every $r>0$, there exists an $R = R(r)$ such that
\[
\liminf_{k \to \infty} \tilde{\bbP}^{\xi_k}_{\Lambda_k(\zero),n,x} (S_{r,R}(\mathbf{f}_k)) >0.
\]
Let $\calV(f)$ be the number of faces surrounded by the outermost loop that surrounds $f$. The above inequality implies that, for all $r$ positive,
\[
\liminf_{k \to \infty} \tilde{\bbP}^{\xi_k}_{\Lambda_k(\zero),n,x} \left(\calV(\mathbf{f_k}) > |\{\text{faces of } \Lambda_r(\zero) \}|\right) >0.
\]

Next, fix a loop configuration $\omega$, and let $\mathsf{Out}(\omega)$ be the set of all outermost loops in $\omega$, and let $\mathsf{Area}(\ell)$ be the number of faces surrounded by an outermost loop $\ell \in \mathsf{Out}(\omega)$. By isoperimetric considerations, there exists an absolute constant $c$ so that $\mathsf{Area}(\ell) \leq c |\ell|^2$.  Then
\begin{align}
	\sum_{\ell \in \mathsf{Out}(\omega)} \mathsf{Area}(\ell) \cdot \1_{\mathsf{Area}(\ell) >r}
	& \leq c \sum_{\ell \in \mathsf{Out}(\omega)} |\ell|^2 \cdot \1_{|\ell| >(r/c)^{1/2}} \nonumber \\
	& \leq   6c \sum_{f } \sum_{\ell \in \mathsf{Out}(\omega)} |\ell| \cdot \1_{f \text{ borders } \ell, \, |\ell| >(r/c)^{1/2}} \nonumber \\
	& \leq   6 c \sum_{f  } \sum_{\ell \in \omega } |\ell| \cdot \1_{f \text{ borders } \ell, \, |\ell| >(r/c)^{1/2}}, \label{eq:unif-int-of-area-ell}
 \end{align}
where the inequality on the second line follows because the number of faces bordering a loop is greater than its length divided by six, and the inequality in the third line simply sums over a larger set.
Dividing both sides of~\eqref{eq:unif-int-of-area-ell} by the number of faces in $\Lambda_k(\zero)$, we obtain
\[
	\tilde{\bbP}^{\xi_k}_{\Lambda_k(\zero),n,x} \left[\calV(\mathbf{f}_k) > r \right] \leq
	6c\cdot \tilde{\bbE}^{\xi_k}_{\Lambda_k(\zero),n,x} \left(\calL(\mathbf{f}_k) \cdot 1_{\calL(\mathbf{f}_k) > (r/c)^{1/2}}\right).
\]
Taking the $\liminf$ as $k$ goes to infinity, we find that the tail expectation is uniformly bounded from below --- i.e. $\{\calL(\mathbf{f}_k)\}$ is not uniformly integrable.
\end{proof}

\bibliographystyle{amsalpha}
\bibliography{biblicomplete}
\end{document}